\documentclass[11pt, a4paper, leqno]{amsart}

\usepackage[left=5in, right=4in, top=3in, bottom=3in, includefoot, headheight=13.6pt]{geometry}

\usepackage{amsmath,enumerate}
\usepackage{amsfonts}
\usepackage{amssymb}
\usepackage{amsthm}
\usepackage{geometry}
\usepackage{latexsym}
\usepackage{mathrsfs}
\usepackage[all]{xy}

\usepackage[ps2pdf=true,colorlinks]{hyperref}
\usepackage[figure,table]{hypcap}
\hypersetup{
    bookmarksnumbered,
    pdfstartview={FitH},
    citecolor={black},
    linkcolor={black},
    urlcolor={black},
    pdfpagemode={UseOutlines}
}
\makeatletter
\newcommand\org@hypertarget{}
\let\org@hypertarget\hypertarget
\renewcommand\hypertarget[2]{%
  \Hy@raisedlink{\org@hypertarget{#1}{}}#2%
}
\makeatother

\theoremstyle{plain}
\newtheorem{theorem}{Theorem}[section]

\newtheorem{lemma}[theorem]{Lemma}
\newtheorem{corollary}[theorem]{Corollary}
\newtheorem{proposition}[theorem]{Proposition}

\theoremstyle{definition}
\newtheorem{example}[theorem]{Example}
\newtheorem{question}[theorem]{Question}
\newtheorem{definition}[theorem]{Definition}
\newtheorem{remark}[theorem]{Remark}

\newtheorem{prob}[theorem]{Problem}

\newcommand{\C} {\Bbb C}
\newcommand{\N}{\Bbb N}
\newcommand{\R}{\Bbb R}
\newcommand{\T}{\Bbb T}
\newcommand{\Z}{\Bbb Z}

\newcommand{\bc} {\Bbb C}
\newcommand{\bn}{\Bbb N}

\newcommand{\Balg} {\mathcal{B}}

\newcommand{\Alg} {\mathcal{A}}

\newcommand{\Tt} {\mathcal T}

\newcommand{\Kk} {\mathcal K}
\newcommand{\Aut}{\textup{Aut}}

\newcommand{\ip}[2]{\langle #1,#2 \rangle}

\newcommand{\Hil}{\mathsf{H}}

\newcommand{\semidir}{{\Bbb o}}
\newcommand{\smooth}{\mathcal C^1(A)}

\newcommand{\eps}{\varepsilon}

\newcommand{\ot}{\otimes}

\newcommand{\wt}{\widetilde}

\numberwithin{equation}{section}

\title{On Spectral Triples on Crossed Products arising from Equicontinuous Actions}
\author{Andrew Hawkins}
\address{School of Mathematical Sciences,  University of Nottingham, Nottingham, NG7 2RD, England}
\email{pmxah@nottingham.ac.uk  }

\author{Adam Skalski}
\address{Institute of Mathematics of the Polish Academy of Sciences,
ul.\ \'Sniadeckich 8, 00-956 Warszawa, Poland} \email{a.skalski@impan.pl}

\author{Stuart White}
\address{ School of Mathematics and Statistics, University of
Glasgow, University Gardens, Glasgow, G12 8QW, Scotland}
\email{stuart.white@glasgow.ac.uk}

\author{Joachim Zacharias}
\address{School of Mathematical Sciences,  University of Nottingham, Nottingham, NG7 2RD, England}
\email{joachim.zacharias@nottingham.ac.uk  }


\begin{document}

\maketitle

\begin{abstract}
\noindent
The external Kasparov product is used to construct odd and even spectral triples on crossed products of $C^*$-algebras
by actions of discrete groups which are equicontinuous in a natural sense. When the group in question is $\Z$ this gives another viewpoint on the spectral triples introduced by Belissard, Marcolli and Reihani. We investigate the properties of this construction and apply it to produce spectral triples on the Bunce-Deddens algebra arising from the odometer action on the Cantor set and some other crossed products of AF-algebras.
\end{abstract}

\section{Introduction}

One of the most fundamental concepts of noncommutative geometry is the notion of spectral triple or unbounded Fredholm module on a  $C^*$-algebra (\cite{Co2}). The key example is given by the Dirac operator on a compact spin manifold $M$, which defines a
spectral triple on the commutative $C^*$-algebra of continuous functions on $M$. The geometry of the manifold is encoded in this `Dirac'-triple. First of all, the triple defines a $K$-homology class, the fundamental class of $M$. Next it defines a metric on the state space of the algebra from which we can recover the metric on $M$: point evaluations are pure states and the distance between
two point evaluations coincides with the geodesic distance of the two points. Further, the dimension of $M$ can be read
off from summability properties of the triple, using the classical Weyl asymptotics for the eigenvalues of the
Laplace operator. All of this can be generalized to noncommutative algebras and a good part of the extensive research in noncommutative geometry
has been devoted to developing the corresponding formalism. Via Kasparov's $KK$-theory, a
spectral triple on a general $C^*$-algebra defines a $K$-homology
class on $A$ (provided mild regularity conditions are met).
Correspondingly, there are two types of spectral triples, even and odd ones (this distinction is more important from the
point of view of $K$-homology, since every even triple gives an odd one by forgetting the grading). A spectral triple defines a
pseudometric on the state space of the algebra. Rieffel (\cite{Rieff})  identified  conditions under which this pseudometric is a
metric inducing the weak-$^*$ topology: the function $ a \mapsto\| [D,a]\|$ has to be a Lip-norm on $A$. If this is true, we will say that the triple satisfies
the \emph{Lip} or \emph{Lipschitz condition} and the corresponding metric will be said to be a \emph{Lip} or \emph{Lipschitz} metric.

Although the definition of spectral triples is at least 20 years old and its predecessors have been around for more than 30 years, so far most of the research has focused on properties of a particular given spectral triple and its associated differential calculus, with the constructions mainly produced on a case by case basis.
Little seems to be known about the existence of spectral triples on general $C^*$-algebras.
Besides the fundamental example of the irrational rotation algebra, `good' spectral triples have
been constructed  for some specific classes of $C^*$-algebras, e.g.\  for group $C^*$-algebras of discrete hyperbolic groups (\cite{OR}), for AF-algebras (\cite{CI}) and for algebras arising as $q$-deformations of the function algebras of simply connected simple compact Lie groups (\cite{NeshTus}).  There are very few general results in literature treating other
important classes, e.g.\ A$\T$-algebras (in \cite{Expo} it was shown that spectral triples always exist on quasidiagonal $C^*$-algebras, but very little can in general be said about their regularity).

In \cite{cmrv,bmr} spectral triples are constructed on $C^*$-algebras arising as crossed products by a $\Z$-action, starting from a triple on the coefficient algebra, provided a certain uniform boundedness condition is satisfied.  In this paper we further examine this construction, showing how to extend it to actions of an arbitrary finitely generated discrete group $\Gamma$. We also develop the investigation of the uniform boundedness condition, showing that it is closely related to metric equicontinuity of the induced transformation on the state space of the algebra.  In the case $\Gamma=\Z$, \cite{bmr} shows that this construction behaves well: if the original triple on the coefficient algebra induces a Lip-metric, then so does the crossed product triple. In section 2, we show how this result can be obtained using a cut-down procedure of Ozawa and Rieffel, and explain the difficulties involved in extending this beyond $\Z$. We also observe that the construction can be iterated in a natural way and use it to prove the existence of Lip-metric inducing spectral triples on crossed products by equicontinuous actions of $\Z^d$.

From our point of view the idea behind this  construction is to use the external unbounded Kasparov product which is designed to produce a spectral triple on a tensor product and check under what conditions the same formula (c.f.\:\cite{Co2}, p.\,434) still defines a triple on a crossed product.  The other ingredient is given by a translation bounded function on the group. If $\Gamma = \Z$ and the translation bounded function is the identity on $\Z$, then the construction gives precisely the spectral triple investigated in  \cite{cmrv} and \cite{bmr}. In this case the even $K$-homology class induced by the triple may be viewed as the image of the odd $K$-homology class defined by the spectral triple on the coefficient
algebra under the boundary map in the Pimsner-Voiculescu sequence (\cite{pv}). Similarly starting from an even spectral triple on the coefficient algebra we give a formula at the end of Section 2 producing an odd triple on the crossed product, which for $\Gamma = \Z$ and
the identity translation bounded function represents the image of the corresponding even $K$-homology class under the other PV-boundary map.

In Section 3 we apply the construction to certain examples of metrically equicontinuous (in fact often even isometric) actions, mostly on AF-algebras.
First we observe that for actions on commutative $C^*$-algebras the construction of spectral triples works only if the action  is equicontinuous on the spectrum of the algebra. The classification of minimal equicontinuous homeomorphisms on the Cantor set
(\cite{Kur} Thm.4.4) shows that the homeomorphisms in this class coincide with the ones which lead to Bunce-Deddens algebras, realised as crossed
products of minimal (odometer) actions on the Cantor set. Using the triples of Christensen and Ivan as the starting ingredient on the coefficient algebra we thus obtain spectral triples on Bunce-Deddens algebras and it follows that these are unique $C^*$-algebras
associated to minimal Cantor systems which fit into the framework studied in the paper. Next  we discuss some basic facts concerning the metrically equicontinuous $\mathbb{Z}$-actions on noncommutative AF-algebras, showing that product type automorphisms belong to this class. Finally we provide a proof of a folklore statement that a minimal action of a countable discrete group $\Gamma$ on the Cantor set is equicontinuous if and only if it is conjugate to a $\Gamma$-subodometer.
Thus on one hand the uniform boundedness condition which we require to construct a spectral triple on the crossed product places strong restrictions on the
possible actions, and on the other we can use the construction to provide spectral triples on generalized Bunce-Deddens algebras investigated in \cite{Orf}. Note that most of the crossed product $C^*$-algebras analysed in Section 3 fall in the A$\mathbb{T}$-class.

All Hilbert spaces and $C^*$-algebras in this paper are assumed to be separable in norm, and algebras are
assumed to be unital unless specified otherwise.

\section{Spectral triples on crossed products}\label{sec2}

We begin this section by recalling the basic definitions related to spectral triples and their properties, and reminding the reader how the Kasparov type product
can be used to produce triples on the tensor product of $C^*$-algebras. Then we analyse a condition on the action of a discrete group $\Gamma$ on a $C^*$-algebra  (and identify it with metric equicontinuity of the action induced on the state space of the algebra), which allows us to use a similar idea to construct
(even) spectral triples on the corresponding crossed products. Next we show that for $\Gamma=\Z$ the last construction preserves the Lipschitz property. Finally we explain how an analogous method can be used to obtain an odd triple on the crossed product, starting from an even triple on the
coefficient algebra. This enables us to
iterate the construction and obtain triples inducing Lipschitz metrics
from equicontinuous actions of $\Z^d$.

\subsection{Spectral triples and compact quantum metric spaces}

\begin{definition} \label{spectrip}
Let $A$ be a separable unital $C^*$-algebra. An \emph{odd spectral triple} $(\Alg, \Hil,D)$ on $A$ consists of a representation
$\pi$ of $A$ on a Hilbert space $\Hil$, a dense $^*$-subalgebra $\Alg$ of $A$ and a densely defined self-adjoint operator $D$
on $\Hil$ such that:
\begin{enumerate}
\item
$(1+D^2)^{-1/2}$ is compact, i.e. $D$ is diagonalizable, has finite dimensional eigenspaces and, if $\Hil$ is infinite dimensional, then the absolute values of the eigenvalues of $D$ increase to infinity.
\item For all $a\in\Alg$,
$\pi(a)$ maps the domain of $D$ into itself and $[D,\pi(a)]$ extends to a bounded operator on $\Hil$.
\end{enumerate}
Such an operator $D$ is referred to as a \emph{Dirac operator}. In many cases it is natural to assume that $\Alg$ is in fact equal to the  \emph{Lipschitz algebra} $\smooth$, i.e.\ the collection of all those $a\in A$ for which the domain of $D$ is invariant under the multiplication by $\pi(a)$ and the operator $[D,\pi(a)]$ is bounded. The operator $a\mapsto [D,\pi(a)]$ can be viewed as a densely defined closable derivation from $A$ to $\mathbb B(\Hil)$, with domain of its closure being $\smooth$. Equipped with the natural norm $\|a\|_1:=\|a\| + \|[D,\pi(a)]\|$ the $*$-algebra $\smooth$ is a Banach algebra (see for example Lemma 1 in \cite{bmr}). Occasionally, when we want to stress the choice of the representation, we denote the triple by $(\Alg,\pi,D)$.

  An \emph{even spectral triple} on $A$ is given
by the same data and  the additional structure of a $\Z_2$-grading, i.e.\ a Hilbert space decomposition
$\Hil=\Hil_0 \oplus \Hil_1$ with respect to which $\pi$ and $D$ decompose as
\begin{equation}
\pi=\pi_0 \oplus \pi_1=
\begin{bmatrix}
\pi_0 & 0 \\
0 & \pi_1
\end{bmatrix},\qquad D=
\begin{bmatrix}
0 & D_0 \\
D_0^* & 0
\end{bmatrix}.
\end{equation}
For $p>0$, the triple is said to be \emph{$p$-summable}  if $(1+D^2)^{-p/2}$ is a trace class operator.
\end{definition}

Rieffel also considered the more general setting of \emph{Lipschitz seminorms} (\cite{R.MetricStateActions}, \cite{R.Metrics},
\cite{R.Vec.Grom.Haus}). There are various versions of this concept. For us it is a seminorm $L: \Alg \to \R_+$,
where  $\Alg$ is a dense subspace of $A$ containing $1$ such that $L(1)=0$.
Such a seminorm $L$ is said to be \emph{lower semicontinuous} (respectively, \emph{closed}) if $\{ a \in \Alg : L(a) \leq r\}$ is closed in $\Alg$ (respectively, $A$) for some and hence all $r>0$. If $L$ is lower semicontinuous then, using a corresponding Minkowski functional, $L$ can be extended to a closed Lipschitz seminorm $\bar{L}$ (cf.\ \cite[Proposition 4.4]{R.Metrics}). One can define a norm on $\Alg$ by $\|a\|_1:=\|a\| + L(a)$ and completing $\Alg$ in this norm gives a subspace $\mathcal C^1(A,L)$ of $A$, which is identical with the domain of $\bar{L}$ when $L$ is lower semicontinuous.

A Lipschitz seminorm induces a pseudometric on the state space, $S(A)$, of $A$ by
\begin{equation}\label{DefD}
d(\omega_1,\omega_2)=\sup \{ |\omega_1(a) - \omega_2(a)| : a \in \Alg, L(a)\leq 1 \}
\end{equation}
A spectral triple on $A$ induces a Lipschitz seminorm on $\Alg$ by the formula
\[ L_D(a)=\|[D,\pi(a)]\|, \;\;\; a \in \Alg .\]
This seminorm is always lower semicontinuous (\cite[Proposition 3.7]{R.Metrics}) and if $L_D(a)$ is defined for all $a \in \smooth$ it is then closed. Moreover, $d_L=d_{\bar{L}}$ (again by \cite[Proposition 4.4]{R.Metrics}), but the metric $d$ might still depend on the choice of $\Alg$ and we will denote it by $d_{\Alg}$  when stressing this dependence.

 As noted in \cite[Section 2 (just prior to Definition 2.1)]{R.Memoir} the pseudometric is unchanged by additionally demanding
that $a=a^*$ in the supremum above.  Since its introduction by Connes in \cite{Co1} this pseudometric has been
extensively studied by Rieffel and other authors (\cite{OR}, \cite{R.MetricStateActions}, \cite{R.Metrics},
\cite{Rieffgroup}, \cite{Rieff}, \cite{R.Memoir}, \cite{Co2}) and an array of examples has been produced. Many properties
of the metric space $(S(A),d)$ can be read out directly from the properties of the triple. In particular Rieffel has focused
investigations on those triples for which the topology induced by the pseudometric is the weak-$^*$ topology on
the state space on $A$. This condition is naturally satisfied by the so-called Monge-Kantorovich distance on the
space of probability measures on a compact space $X$ (viewed as the state space of $C(X)$).  Rieffel
characterised when this property holds in the following fashion. (See \cite{R.Metrics}, where this is in fact done in the more general context of metrics arising from Lipschitz seminorms.)
\begin{enumerate}
\item The pseudometric $d$ separates points of the state space if and only if the triple is
\emph{non-degenerate}, i.e.\ the representation $\pi$ is faithful and, for $a\in\Alg$, $[D,\pi(a)]=0$ if and only
if $a\in\mathbb C1$.
\item The pseudometric $d$ is a metric on the state space of $A$ with finite diameter if and
only if the triple is non-degenerate and the image of  $\{a\in\Alg:\|[D,\pi(a)]\|\leq 1\}$ is bounded in the
quotient space $A/\mathbb C1$.
\item The pseudometric is a metric inducing the weak-$^*$ topology on the state
space of $A$ if and only if the triple is non-degenerate and the image of $\{a\in\Alg:\|[D,\pi(a)]\|\leq 1\}$ is
totally bounded in the quotient space $A/\mathbb C1$.  In this case $d$ is known as a \emph{Lip-metric} and the triple gives
us a \emph{spectral quantum metric space} structure on $A$. We also say then that the triple $(\Alg,\Hil,D)$ satisfies the \emph{Lipschitz condition}. Recall that if $L$ is a Lipschitz seminorm, not necessarily coming from a Dirac operator, such that the corresponding metric on the state space is compatible with the weak-$^*$ topology, then the pair $(A,L)$ is called a   \emph{(compact) quantum metric space}.
\end{enumerate}

A priori all these concepts could depend on the choice of the $C^*$-norm dense algebra $\Alg$.
We are grateful to Adam Rennie for pointing out that the first of them does not.  Let $(P_n)_n$ be the spectral projections of the Dirac operator.
Non-degeneracy is equivalent to $\Alg\cap\{P_n\}'=\mathbb C1$ and, by the double commutant theorem, it is also
equivalent to $(A'\cup\{P_n\})''=\mathbb B(\Hil)$; the last equality is explicitly independent of the choice of the dense subalgebra
$\Alg$.

The situation with the Lipschitz condition seems to be more complicated. If the triple is nondegenerate and $d_{\smooth}$ is a Lip-metric, then so is $d_{\Alg}$ for any other choice of a dense subalgebra $\Alg$ contained in $\smooth$. Using the results in Section 4 of \cite{R.Metrics} one can show that given two such subalgebras $\Alg_1$ and $\Alg_2$ the metrics $d_{\Alg_1}$ and $d_{\Alg_2}$ are equal if and only if the sets $\{a\in \Alg_1: \|[D,\pi(a)]\|\leq 1\}$  and $\{a\in \Alg_2: \|[D,\pi(a)]\|\leq 1\}$ have equal ($C^*$-norm) closures in $A$. The latter condition is trivially satisfied if the closures of $\Alg_1$ and $\Alg_2$ in the Lipschitz norm $\|\cdot\|_1$ of $\smooth$ are equal. We do not know however the answers to the following questions, the second of which can be viewed as  a natural problem related to norming properties of subsets of a classical Banach space.

\begin{prob}
Does there exist a nondegenerate spectral triple $(\Alg,\Hil,D)$ on a $C^*$-algebra $A$ such that $d_{\Alg}$ is a Lip-metric and $d_{\smooth}$ is not?
\end{prob}

\begin{prob}
Does there exist a uniformly dense subalgebra $\Alg$ of $C(\mathbb{T})$ consisting of Lipschitz functions and such that the metric on the probability measures on $\mathbb{T}$ given by the formula:
\[ d_{\Alg}(\mu, \nu) = \sup\{|\mu(f) - \nu(f)|: f \in \Alg,\ L(f)\leq 1\}\]
is strictly smaller then the usual bounded Lipschitz distance (known by the theorem of Kantorovich and Rubinstein to coincide with the celebrated Wasserstein distance $W_1$, see \cite{Villani})
\[ d(\mu, \nu) = \sup\{|\mu(f) - \nu (f)| : f \in C(\mathbb{T}),\ L(f)\leq 1\}?\]
Note that the metric $d_{\Alg}$ necessarily equips the space of the probability measures on $\mathbb{T}$ with the weak-$^*$ topology.
\end{prob}

In \cite{Co1}, Connes provided examples of spectral triples on the reduced group $C^*$-algebras of discrete groups arising from proper translation bounded functions.  These triples provide an important ingredient in the construction of spectral triples on reduced crossed products.
\begin{example}\label{GroupTriple}
Let $\Gamma$ be a discrete group equipped with a \emph{proper translation bounded function} $l:\Gamma\rightarrow\Z$, i.e.
\begin{enumerate}
\item $l(g)=0$ if and only if $g$ is the identity element of $\Gamma$;
\item The level sets $\{g\in\Gamma:l(\gamma)=n\}$ are finite for each $n\in\Z$;
\item For all $g\in\Gamma$, the translation function $l_g:\Gamma\rightarrow\Z$ given by $l_g(x)=l(x)-l(g^{-1}x)$ is bounded.
\end{enumerate}
Let $M_l$ be the usual self-adjoint extension of the unbounded operator of multiplication by $l$ on
$\ell^2(\Gamma)$ with domain consisting of the finitely supported elements of $\ell^2(\Gamma)$. In this way we
obtain a non-degenerate odd spectral triple $(\mathbb C\Gamma,\ell^2(\Gamma),M_l)$ on the reduced group
$C^*$-algebra $C^*_r(\Gamma)$.  When $\Gamma$ is finitely generated, the word length function with respect to a
fixed set of generators gives an example of a proper translation bounded function. Rieffel showed that this
triple turns $C^*_r(\Gamma)$ into a compact quantum metric space when $\Gamma=\Z^d$ is equipped with any length
function arising from a finite generating set (\cite{Rieffgroup}) and Ozawa and Rieffel established the same result when $\Gamma$ is a hyperbolic group
(\cite{OR}). Another important example of a proper translation bounded function is the identity function $\iota$
on $\Z$. In this case the triple we obtain on $C^*_r(\Z)\cong C(\mathbb T)$ agrees with the usual triple on
$C(\mathbb T)$ arising from the Dirac operator $\frac{1}{i}\frac{\mathrm{d}}{\mathrm{d}\theta}$ and, at the level of
$K$-homology, represents a generator of $K^1(C^*_r(\Z))\cong \Z$.
\end{example}

Odd and even spectral triples give rise to odd and even $K$-homology classes respectively. Moreover, given two
odd triples $(\Alg,\Hil_A,D_A)$ and $(\Balg,\Hil_B,D_B)$ on $C^*$-algebras $A$ and $B$ there is a formula for
producing an even triple on the spatial tensor product $A\otimes B$. Let $\Hil=\Hil_A\otimes\Hil_B$, and define a
Dirac operator $D$ on $\Hil\oplus\Hil$ by
\begin{equation}\label{DefDirac}
D=
\begin{bmatrix}
0 & D_A \otimes 1 - i \otimes D_B \\
D_A \otimes 1 + i \otimes D_B & 0
\end{bmatrix}.
\end{equation}
Then $(\Alg\odot \Balg,\Hil\oplus\Hil,D)$, where $\Alg\odot \Balg$ denotes the
algebraic tensor product of $\Alg$ and $\Balg$, is an even spectral triple on $A\otimes B$ (c.f.\:\cite{Co2}, p.\,434). One can directly describe the eigenvalues and eigenspaces of $D$ in terms of those of $D_A$ and $D_B$ to see that $D$ is self-adjoint. This calculation also shows that if $D_A$ and $D_B$ are $p$ and $q$-summable respectively, then $D$ is $(p+q)$-summable as if $((1+\lambda_m^2)^{-p/2})_{m=1}^{\infty}$ and $((1+\mu_n^2)^{-q/2})_{n=1}^{\infty}$
are summable then the double sequence $((1+\lambda_m^2+\mu_n^2)^{-(p+q)/2})_{n,m=1}^{\infty}$ is summable. The last fact follows from the inequality
\begin{equation}
\frac{1}{\left(x+y -1\right)^{\alpha + \beta}} \leq \frac{1}{x^{\alpha}y^{\beta}}
\end{equation}
valid for all $x,y >1, \alpha, \beta>0$.

The triple $(\Alg \odot \Balg,\Hil\oplus\Hil,D)$ corresponds to a certain Kasparov product of the respective
$K$-homology classes. The main observation in this paper is that the operator $D$ can also be used to define a
triple on reduced crossed products $A\rtimes_{r,\alpha}\Gamma$ provided the action is metrically equicontinuous as defined below. This
gives us a different way of viewing the spectral triples on the crossed products $A\rtimes_\alpha\Z$ recently introduced in \cite{bmr}, as explained in the next subsection.

\subsection{Metrically equicontinuous actions and construction of spectral triples on crossed products}

Recall that a continuous  action of $\Gamma$ on a compact metric space $(X,d)$ is \emph{equicontinuous} if for all $\varepsilon >0$ there exists $\delta >0$ such that $d(x,y) < \delta $ implies
$d(gx,gy) < \eps $ for all $g \in \Gamma$. Note that this means that $\bar{d}(x,y):=\sup_g d(gx,gy)$ is an invariant
metric inducing the same topology, and that equicontinuity is characterised be the existence of such an invariant metric. However, the two metrics $d$ and $\bar{d}$ need not be equivalent, that is, there need not exist any $C > 0$ such that $\frac{1}{C}d(x,y) \leq \bar{d}(x,y) \leq Cd(x,y)$ for all $x,y \in X$. When there is an equivalent invariant metric, we will call the action \emph{metrically equicontinuous}; easy examples show that this condition is in general indeed stronger than equicontinuity.

Note that an action $\alpha : \Gamma \to \textup{Aut} (A)$ defines a continuous  action $\hat{\alpha}$ of $\Gamma $
on the compact metrisable space $S(A)$ via $  \hat{\alpha}_g(\phi)= \phi \circ \alpha_{g^{-1}}$.

\begin{definition} \label{equicont} Let $A$ be a unital $C^*$-algebra with a dense unital $^*$-subalgebra $\Alg$ and a
Lipschitz seminorm $L:\Alg\to \mathbb{R}_+$ such that $(A,L)$ is a compact quantum metric space. An action $\alpha : \Gamma \to \textup{Aut} (A)$ is called \emph{smooth} if for each $g\in \Gamma$ the automorphism $\alpha_g$ leaves $\Alg$ globally invariant. If  moreover
\begin{equation}\label{Hypothesismetric}
\sup_{g\in \Gamma}L(\alpha_g(a))<\infty,\quad a\in\Alg,
\end{equation}
then we will call the action $\alpha$ \emph{metrically equicontinuous}. Finally if a smooth action $\alpha$ is such that
$L(\alpha_g(a))=L(a)$ for all $a\in\Alg$ and $g\in\Gamma$, then we will call it \emph{isometric}.
\end{definition}

If $(A,L)$ is a compact quantum metric space and $\alpha : \Gamma \to \textup{Aut} (A)$ is metrically  equicontinuous, it will be convenient to introduce the Lipschitz seminorm $L_{\Gamma}:\Alg \to \mathbb{R}_+$,
\begin{equation}
L_{\Gamma}(a) := \sup_{g\in \Gamma}L(\alpha_g(a)),\quad a\in\Alg,
\end{equation}
which in turn determines a pseudometric $d_{L_{\Gamma}}$ on $S(A)$. An elementary check reveals that since $\{a \in \Alg: L(a) \leq 1\}$ separates the states of $A$, $\{a \in \Alg: L_{\Gamma}(a) \leq 1\}$ also separates the states of $A$ and that $d_{L_{\Gamma}}$ is a metric, with $d_{L_{\Gamma}} \leq d_L$ pointwise. The map $\iota: (S(A),d_L) \mapsto (S(A),d_{L_{\Gamma}})$ is a continuous map of metric spaces with $(S(A),d_L)$ compact, whence a homeomorphism. Since the induced action of $\Gamma$ on the metric space $(S(A),d_{L_{\Gamma}})$ is actually \textit{isometric}, it follows that $\hat{\alpha}$ implements an equicontinuous action of $S(A)$, when equipped with Connes' metric for the compact quantum metric $(A,L)$.

The next result shows that if $L:\Alg\to \mathbb{R}_+$ is closed, then metric equicontinuity of the action $\alpha:\Gamma\rightarrow\textup{Aut}(A)$ is equivalent to metric equicontinuity of the induced action on the state space and also to the existence of an equivalent Lipschitz seminorm for which the $\alpha$ acts by isometries. This justifies our terminology.  We will subsequently use our construction of a spectral triple on crossed products to show that if $\alpha$ is a metrically equicontinuous action on spectral quantum metric space $(A,D)$, then we can construct an equivalent spectral quantum metric $(A,D')$ for which $\alpha$ is isometric (Corollary \ref{spectralqms}).

\begin{proposition}\label{equicontthm}
Let $A$ be a unital $C^*$-algebra with a dense unital $^*$-subalgebra $\Alg$ and a lower semicontinuous Lipschitz seminorm $L:\Alg\to \mathbb{R}_+$ (in particular $\Alg = \smooth$) such that $(A,L)$ is a compact quantum metric space.
Assume that $\alpha:\Gamma\to \textup{Aut}(A)$ is a smooth action. Consider the following statements:
\begin{enumerate}[(i)]
\item $\alpha$ is metrically equicontinuous on $(A,L)$;
\item there exists a Lipschitz seminorm $L'$ on $\Alg$ for which $\alpha$ is isometric and a constant $C>0$ such that $L(a)\leq L'(a)\leq CL(a)$ for all $a\in\Alg$;
\item there exists an $\hat{\alpha}$-invariant metric $d$ on $S(A)$ which is equivalent to $d_L$.
\end{enumerate}
Then (ii)$\Longrightarrow$(iii)$\Longrightarrow$(i) in general, and in the case that $L$ is closed we additionally have (i)$\Longrightarrow$(ii).
\end{proposition}
\begin{proof}
Suppose first that (i) holds so that $\alpha$ is metrically equicontinuous for the compact quantum metric
space $(A,L)$ and suppose that $L$ is closed. Then $\|a\|_{1,\Gamma} := \;\|a\| + L_{\Gamma}(a)$  is a new norm on $\Alg$ dominating $\|a\|_1= \| a \| + L(a)$. Since $L$ is closed, $(\Alg,\| \cdot \|_1)$ is a Banach space, and as in \cite[Corollary 1]{bmr}, $\Alg$ is also complete in the norm $\|\cdot\|_{1,\Gamma}$. For completeness, we provide a proof of this fact. If $(a_n)_{n=1}^{\infty}$ is a Cauchy sequence with respect to \  $\| \cdot \|_{1,\Gamma}$
then it is also Cauchy with respect to $\|\cdot\|_1$, so converges to some $a \in \Alg$. In particular $L(a_n-a) \to 0$ and
$\|a_n -a\| \to 0$ as $n \to \infty$. Let $\eps >0$ then there is $n_0 \in \N$ such that
$\|a_n - a_m\|_{1,\Gamma} < \eps/4$ (and so $\|a_n - a_m\| < \eps/4$ and $L(\alpha_g (a_n - a_m)) <\eps/4$ for all $g \in \Gamma$) whenever $n,m \geq n_0$. Since $\alpha_g(a_n) \to \alpha_g(a)$ and $L$ is lower semicontinuous it follows that
for $m\geq n_0$, $L(\alpha_g (a - a_m)) \leq \eps/4 $ for every $g \in \Gamma$. Hence
$\sup_{g \in \Gamma}L(\alpha_g(a - a_m)) \leq \eps/4  < \eps/2 $ and
$\|a - a_m\|_{1,\Gamma}  <\eps $
whenever $m \geq n_0$.  Thus $(a_n)_{n=1}^{\infty}$ converges to $a$ in $(\Alg , \| \cdot \|_{1,\Gamma})$.
By the open mapping theorem it follows that
$\| \cdot \|_1$ and $ \| \cdot \|_{1,\Gamma}$ must be equivalent norms on $\Alg$.
In particular there is a constant $K > 0$ such that
$$
L_{\Gamma}(a) \leq K(\|a\| + L(a)),\quad a \in \Alg.
$$
Since $L(a) = L_{\Gamma}(a) = 0$ whenever $a$ is a multiple of the identity of $A$, this inequality passes to the quotient algebra $A / \C I$, where it reads
$$
L_{\Gamma}(\tilde{a}) \leq K(\|\tilde{a}\|_{\Alg / \C I} + L(\tilde{a})),\quad \tilde{a} \in \Alg / \C I.
$$
By \cite[Lemma 2.1]{R.Metrics} we may identify $\{ \phi -\psi : \phi , \psi \in S(A) \}$ with $D_2=\{ \lambda \in (A/\C I)^* : \| \lambda \| \leq 2 , \lambda^* = \lambda \}$, where $\lambda^*(x)=\overline{\lambda(x^*)}$.   Moreover as in \cite{R.Metrics}, Proposition 2.2 we have for $\lambda =  \phi -\psi \in D_2$
$$
|\lambda (a) | =  |\phi (a)-\psi (a) | \leq L(a) d_L(\phi ,\psi)
$$
and since $\|\tilde{a}\|_{\Alg / \C I} \leq \sup \{ |\lambda (a) | : \lambda \in D_2 \}$ it follows that
$$
\| \tilde{a} \|_{\Alg / \C I} \leq L(a)  \textup{diam}_{d_L}(S(A))
$$
for all $a \in \Alg / \C I$. Since $(A,L)$ is a compact quantum metric space it must have finite diameter in the metric $d_L$. Thus for all $a \in \Alg$
\begin{eqnarray*}
L_{\Gamma}(a) &=&  L_{\Gamma}(\tilde{a}) \\
& \leq & K(\|\tilde{a}\|_{\Alg / \C I} + L(\tilde{a})) \\
& \leq & K\big(\textup{diam}_{d_L}(S(A)) + 1\big) L(\tilde{a}) \\
& = & K\big(\textup{diam}_{d_L}(S(A)) + 1\big) L(a).
\end{eqnarray*}
Setting $C = K(\textup{diam}_{d_L}(S(A)) + 1)$ we obtain the estimate
\begin{eqnarray}
L(a) \leq L_{\Gamma}(a) \leq C L(a), \;\;a \in \Alg.
\end{eqnarray}

For the implication (ii)$\Longrightarrow$(iii), given a Lipschitz seminorm $L'$ on $\Alg$ and $C>0$ such that $L(a)\leq L'(a)\leq CL(a)$ for all $a\in\Alg$ so that $\alpha$ acts isometrically on $(A,L')$, note that the metric $d_{L_{\Gamma}})$ on $S(A)$ is invariant for $\Gamma$, Further
\begin{eqnarray}
\frac{1}{C} d_L(\phi,\psi)  \leq d_{L_{\Gamma}}(\phi,\psi) \leq  d_L(\phi,\psi) \leq C d_L(\phi,\psi) , \;\;\phi,\psi \in S(A)
\end{eqnarray}
so that (iii) holds.

To show (iii)$\Longrightarrow$(i)  let $\alpha: \Gamma \to \textup{Aut}(A)$ be such that $\hat{\alpha}$ is equicontinuous for $(S(A), d_L)$ and suppose that $d$ is another metric with $\frac{1}{C}d \leq d_L \leq Cd$ (pointwise) and such that $\Gamma$ implements an isometric action on $(S(A),d)$. As $L$ is lower semicontinuous, we may recover $L$ uniquely from $d_L$ by \cite[Theorem 4.1]{R.Metrics}, at least on the restriction of $L$ to $\Alg_{sa}= \{ a \in \Alg : a^*=a \}$, via
$$
L(a) = \sup \bigg{ \{ } \frac{ |\phi(a) - \psi(a)|} {d_L(\phi, \psi)} : \;\phi, \;\psi \in S(A), \;\phi \neq \psi \bigg{ \} },\quad a\in\Alg.
$$
The same formula with $d_L$ replaced by $d$ defines a $\Gamma$-invariant Lipschitz seminorm $L'$ on the same domain $\Alg$ as $L$. (Note that we are using $\alpha_g(\Alg) \subset \Alg$ for $g \in \Gamma$ here.)
It is easy to see that  $\frac{1}{C}L'(a) \leq L(a) \leq CL'(a)$ whenever $a \in \Alg$; moreover $d_{L'}=d$ because as mentioned before the definition of $d_{L'}$ only requires selfadjoint elements.
Since $L'$ is $\Gamma$-invariant the action of $\Gamma$ is smooth and isometric on  $(A,L')$ thus the action on $(A,L)$ is metrically equicontinuous. This concludes the proof.
\end{proof}


We now return to the discussion of the crossed products. Suppose that $\alpha$ is an action of a discrete group $\Gamma$ on a unital $C^*$-algebra $A$ and suppose that $A$ is equipped with an odd spectral triple $(\Alg,\Hil_A,D_A)$ via a faithful representation $\pi$ of $A$ and $C^*_r(\Gamma)$ is equipped with an odd spectral triple $(\mathbb C\Gamma,\ell^2(\Gamma),M_l)$ arising from a proper translation bounded function $l:\Gamma\rightarrow\Z$ as described in Example \ref{GroupTriple}. There is a covariant representation $(\tilde{\pi},\lambda)$ of $(A,\Gamma,\alpha)$ on $\Hil=\Hil_A\otimes\ell^2(\Gamma)$ given by
\begin{equation}\label{DefDirac1}
\tilde{\pi}(a)(\xi\otimes \delta_g)=\pi(\alpha_{g^{-1}}(a))\xi\otimes\delta_g,\quad
\lambda_h(\xi\otimes\delta_g)=\xi\otimes\delta_{hg},\qquad a\in A,\ \xi\in\Hil_A,\ g,h\in \Gamma.
\end{equation}
We can use this to view the reduced crossed product $A\rtimes_{r,\alpha}\Gamma$ as a subalgebra of $\mathbb B(\Hil)$. Just as in the tensor product situation we define a  Dirac operator $D$ on $\Hil\oplus\Hil$ by
\begin{equation}\label{DefDirac2}
D=
\begin{bmatrix}
0 & D_A \otimes 1 - i \otimes M_l \\
D_A \otimes 1 + i \otimes M_l & 0
\end{bmatrix}.
\end{equation}
Write $C_c(\Gamma,\Alg)$ for the natural dense $^*$-subalgebra of $A\rtimes_{r,\alpha}\Gamma$ consisting of those finite sums $\sum_{g\in\Gamma}x_g\lambda_g$, where each $x_g\in\Alg$. Since we regard $A\rtimes_{r,\alpha}\Gamma$ as represented on $\mathbb B(\Hil)$, we write $[\pi(a),D_A]$ when we want to compute commutators of elements $a\in\Alg$ with $D_A$ on $\Hil_A$. In the theorem below, we do not use the terminology of metrically equicontinuous actions, as we do not assume that $D_A$ satisfies the Lipschitz-condition.

\begin{theorem} \label{constr}
Let $(\mathcal{A}, \Hil_{A}, D_A)$ be an odd spectral triple on a unital $C^*$-algebra A and assume that $\alpha: \Gamma \to \textup{Aut} (A)$ is an action of a discrete group $\Gamma$, equipped with a proper translation bounded function $l$, such that $\alpha_g(\mathcal{A}) \subset \mathcal{A}$ for all $g \in \Gamma$ and
\begin{equation}\label{Hypothesis}
\sup_{g\in \Gamma}\|[D_A,\pi(\alpha_g(a))]\|<\infty,\quad a\in\Alg.
\end{equation}
Then $(C_c (\Gamma, \mathcal{A}), \Hil \oplus \Hil, D)$, with $D$ defined via (\ref{DefDirac2}), is an even spectral triple on $A\rtimes_{r,\alpha}\Gamma$.  If $D_A$ is $p$-summable and $M_l$ is $q$-summable, then $D$ is $(p+q)$-summable.  If $D_A$ is non-degenerate, then so too is $D$.
\end{theorem}
\begin{proof}
As the Dirac operator is given by (\ref{DefDirac2}) it is densely defined and self-adjoint with compact resolvent, just as in the tensor product situation described above.  The summability statement also follows immediately from the same statement for the tensor product triple.  For $\xi$ in the domain of $D_A$, $a\in\Alg$ and $h\in\Gamma$, we have
\begin{equation}
[D_A\otimes 1,a](\xi\otimes\delta_h)=[D_A,\pi(\alpha_{h}(a))]\xi\otimes\delta_{h}.
\end{equation}
Thus the hypothesis (\ref{Hypothesis}) ensures that $[D_A\otimes 1,a]$ is bounded for $a\in\Alg$.  As $[D_A\otimes 1,\lambda_g]=0$ for all $g\in\Gamma$, it follows that $[D_A\otimes 1,x]$ is bounded for all $x\in C_c(\Gamma,\Alg)$. In a similar fashion, direct computations give
\begin{equation}
[1\otimes M_l,a\lambda_h]=(1\otimes M_{l_h})a\lambda_h,\quad a\in\Alg,\ h\in\Gamma,
\end{equation}
where $M_{l_h}$ is the operator of multiplication by the function $l_h(x)=l(x)-l(h^{-1}x)$ on $\ell^2(\Gamma)$. This is bounded as $l$ is translation bounded.  Thus $[D,x\oplus x]$ is bounded whenever $x\in C_c(\Gamma,\Alg)$.

Suppose that $D_A$ is non-degenerate for $A$ and suppose $x=x^*=\sum_{g\in\Gamma}x_g\lambda_g$ is a finite sum with each $x_g\in\Alg$ such that $[D,x\oplus x]=0$. Then, $[D_A\otimes 1,x]=[1\otimes M_l,x]=0$. Since
$$
0=\ip{[D_A\otimes 1,x](\xi\otimes\delta_{g^{-1}})}{\eta\otimes\delta_e}=\ip{[D_A,\pi(x_g)]\xi}{\eta}
$$
for $\xi$ and $\eta$ in the domain of $D_A$, it follows that each $[D_A,\pi(x_g)]=0$ and so each $x_g\in\mathbb C1$, whence $x\in\mathbb C\Gamma$. Since the triple $(\mathbb C\Gamma,\ell^2(\Gamma),M_l)$ is non-degenerate, $x\in\mathbb C1$ and $(C_c(\Gamma, \mathcal{A}),\Hil\oplus\Hil,D)$ is non-degenerate, and this completes the proof.
\end{proof}

Note that when the action $\alpha$ is trivial, the crossed product $A\rtimes_{r,\alpha}\Gamma$ is the spatial tensor product $A\otimes C^*_r(\Gamma)$ and the triple we obtain from Theorem \ref{constr} is just the tensor product triple. The following proposition will be used in the last part of this section, when we discuss the iterates of the above construction.

\begin{proposition}\label{iterate}
Let $A$, $(\mathcal{A}, \Hil_{A}, D_A)$, $\Gamma$ and $\alpha: \Gamma \to \textup{Aut} (A)$ satisfy all the assumptions of Theorem \ref{constr}. Suppose that $\Gamma'$ is another discrete group, $\beta: \Gamma' \to \textup{Aut} (A)$  is an action which commutes with $\alpha$ (i.e.\ for each $g \in \Gamma$, $g' \in \Gamma'$ we have $\alpha_{g} \beta_{g'} = \beta_{g'} \alpha_{g}$), is smooth (i.e.\ $\beta_{g'}(\mathcal{A}) \subset \mathcal{A}$ for all $g' \in \Gamma'$) and such that
\[
\sup_{g'\in \Gamma'}\|[D_A,\pi(\beta_{g'}(a))]\|<\infty,\quad a\in\Alg.
\]
Let $\wt{\beta}$ denote the canonical extension of $\beta$ to the action of $\Gamma'$ on $A\rtimes_{r,\alpha}\Gamma$ and let $(C_c (\Gamma, \mathcal{A}), \Hil \oplus \Hil, D)$ denote the spectral triple constructed in Theorem \ref{constr}. Then
\[
\sup_{g'\in \Gamma'}\|[D,\wt{\beta}_{g'}(b)]\|<\infty,\quad b\in C_c (\Gamma, \mathcal{A}).
\]
\begin{proof}
Elementary check.
\end{proof}

\begin{remark} We are grateful to the referee's comments leading to this remark. In certain situations we can describe our construction of $D$ essentially as an equivariant Kasparov product.
In the situation of Theorem \ref{constr} suppose there exists a unitary action $u: \Gamma \to B(\Hil_{A})$ on  $\Hil_{A}$ implementing $\alpha$ such that $u_g$ leaves the domain of $D_{A}$ invariant for all $g \in \Gamma$ and $[u_g, D_A]$ extends to a bounded operator on $\Hil_{A}$ (equivalently $g(D_A)-D_A$ is bounded for all $g \in \Gamma$).  This means that $(\mathcal{A}, \Hil_{A}, D_A)$ is an unbounded equivariant Kasparov module defining a class in
$KK_1^{\Gamma}(A,\C)$. Then it is easy to see that our assumption (\ref{Hypothesis}) is verified if
$g(D_A)-D_A$ is uniformly bounded in norm for all $g \in \Gamma$. Moreover, a translation bounded function $l:\Gamma \to \R$ defines an unbounded equivariant spectral triple $(\C, M_l , \ell^2(\Gamma))$ giving an element in $KK^{\Gamma}(\C,\C)$, where the group action is given by $\lambda_g$. Using the correspondence between bounded and unbounded Kasparov modules (\cite{Bla}, Section 17)) the external Kasparov product of the equivariant triples $(\mathcal{A}, \Hil_{A}, D_A)$
and $(\C, M_l , \ell^2(\Gamma))$ is given by the same formula as our Dirac operator $D$. However, the action of $A$ is now given by
$a(\xi \otimes \delta_g)= \pi (a)\xi \otimes \delta_g$ and the action of $\Gamma$ by $h(\xi \otimes \delta_g)= u_h\xi \otimes \delta_{hg}$ on both
summands of $\Hil_A \otimes \ell^2(\Gamma)$. The unitary $U(\xi \otimes \delta_g)=u_g \xi \otimes \delta_g$ conjugates this action
into our action (\ref{DefDirac1}) and $D$ into $(U \oplus U)D(U\oplus U)^*$ which by assumption (\ref{Hypothesis}) is a bounded perturbation of $D$, so defines the same
class in $KK^{\Gamma}$. This class in $KK_0^{\Gamma}(A,\C)$ can be regarded as a class in $K^0(A\rtimes_{\alpha,r} \Gamma)$ via the Green-Julg Theorem.

Alternatively, we can think of our construction as the equivariant external product of $(\mathcal{A}, \Hil_{A}, D_A)$ in $KK_1^{\Gamma}(A,\C)$ and
$(\C , M_l, c_0(\Gamma))$ in $KK_1^{\Gamma}(\C,  c_0(\Gamma))$ combined with the standard morphism
$j_{\Gamma} : KK_0^{\Gamma} (A,c_0(\Gamma)) \to KK_0(A \rtimes_{\alpha,r} \Gamma , c_0(\Gamma)\rtimes_{\alpha,r} \Gamma)=KK_0(A \rtimes_{\alpha,r} \Gamma , \C)$ (cf.\;\cite{Bla} 20.6), again in case we start with an equivariant triple.
\end{remark}

\end{proposition}

\subsection{The Lipschitz condition for spectral triples on crossed products}

It is natural to ask whether the crossed product triple constructed in Theorem \ref{constr} turns $A\rtimes_{r,\alpha}\Gamma$ into a compact quantum metric space and in this subsection we discuss when this can be achieved. Note that, restricted to $A$, the crossed product triple always satisfies the Lipschitz condition when $(\Alg,\Hil_A,D_A$) does.  In particular, given an equivariant action on a spectral quantum metric space, this process produces a new spectral quantum metric on the original space $A$ under which the group acts by isometries. This is formulated in the following corollary.
\begin{corollary} \label{spectralqms}
Let $(\mathcal{A}, \Hil_{A}, D_A)$ be an odd spectral triple on a unital $C^*$-algebra $A$ satisfying the Lipschitz condition. Suppose that $\Gamma$ is a discrete group which admits a proper translation bounded function $l$.  Assume that $\alpha: \Gamma \to \textup{Aut} (A)$ is a metrically equicontinuous action of  $\Gamma$ on the spectral quantum metric space $(A,L_{D_A})$. Then there exists an (even) spectral triple $(\mathcal A,(\Hil_A\otimes\ell^2(\Gamma))\oplus (\Hil_A\otimes\ell^2(\Gamma)), D)$ on $A$ satisfying the Lipschitz condition such that $\alpha$ acts by isometries (in the sense of Definition \ref{equicont}) on the spectral  quantum  metric space $(A,L_{D})$.
\end{corollary}
\begin{proof}
It suffices to observe that the Dirac operator constructed in Theorem \ref{constr} can be used to consider the `restricted' spectral triple $(\Alg, \Hil \oplus \Hil, D)$. It is easy to check that for $a\in \Alg$ we have $L_{D} (a)=\sup_{g \in \Gamma} \|[D_A, \pi(\alpha_{g}(a))]\|=\sup_{g\in\Gamma}L_{D_A}(\pi(\alpha_g(a))$. Since $L_D(a)\geq L_{D_A}(a)$ for $a\in\Alg$, $D$ inherits the Lipschitz condition from $D_A$.  As noted before Proposition \ref{equicontthm}, $\alpha$ acts by isometries on $L_D$, completing the proof.
\end{proof}

The pseudometric $d$ induced by $D$ according to (\ref{DefD}) is equivalent to the pseudometric $d_{1}$ defined by
$$
d_{1}(\omega_1,\omega_2)=\sup\{|\omega_1(x)-\omega_2(x)|:x\in C_c(\Gamma,\Alg),\ \|[x,D_A\otimes 1]\|,\ \|[x,1\otimes M_l]\|\leq 1\}.
$$
Indeed, recall that, as with all Lipschitz type pseudometrics arising from Dirac operators, $d_1$ is unchanged by
additionally demanding that $x=x^*$ in the defining supremum above (see \cite[Section 2]{R.Memoir}). The
equivalence of $d$ and $d_1$ now follows, as if $x=x^*\in C_c(\Gamma,\Alg)$ satisfies $\|[x\oplus x,D]\|\leq 1$ then
$\|[x,D_A\otimes 1+i\otimes M_l]\|\leq 1$ and taking real and imaginary parts gives $\|[x,D_A\otimes 1]\|\leq 1$
and $\|[x,1\otimes M_l]\|\leq 1$. Conversely, if $x=x^*\in C_c(\Gamma,\Alg)$ satisfies $\|[x,D_A\otimes 1]\|\leq 1$ and
$\|[x,1\otimes M_l]\|\leq 1$, then $\|[x\oplus x,D]\|\leq 2$. Hence Rieffel's characterisation of Lip-metrics in
\cite{R.Metrics} shows that $d$ is a Lip-metric if and only if $D$ is non-degenerate and the set
$$
\{x\in C_c(\Gamma, \Alg):\|[x,D_A\otimes 1]\|,\ \|[x,1\otimes M_l]\|\leq 1\}
$$
has a totally bounded image in $(A \rtimes_{r,\alpha}\Gamma)/\mathbb C1$.

Like \cite{bmr} we are able to show that $A\rtimes_{r,\alpha}\Gamma$ is a compact quantum metric space in the
case that $\Gamma=\Z$ is equipped with the identity function $\iota$ as a proper translation bounded function.
The key tool we use to do this is a cut-down procedure of Ozawa and Rieffel from \cite{OR} which, in this
instance, enables us to uniformly approximate elements $x\in C_c(\Z,\Alg)$ with $\|[x,1\otimes M_\iota]\|\leq 1$.
After the proof we discuss why we have been unable to get this to work for more general groups. Note that as we
are using  $\iota$ rather than the length function $n\mapsto |n|$, the resulting triple
$(C^*(\Z),\ell^2(\Z),M_\iota)$ on $C^*(\Z)$ is the generator of $K^1(C^*(\Z))$ which corresponds to the usual
Toeplitz-extension
$$
0 \to \Kk \to \Tt \to C(\T) \to 0
$$
and the triple $(C_c(\Z,\Alg),\Hil\oplus\Hil,D)$ to the Kasparov product of the Toeplitz-extension of the crossed
product
$$
0 \to A \otimes \Kk \to \Tt_{\alpha} \to A \rtimes_{\alpha} \Z \to 0
$$
and $(\Alg,\Hil_A,D_A)$.  This means that the triple $(C_c(\Z,\Alg),\Hil\oplus\Hil,D)$ arising in the
construction corresponds to the image of $[(\Alg,\Hil_A,D_A)] \in K^1(A)$ under the boundary map in the
Pimsner-Voiculescu sequence (c.f.\;\cite{pv}, \cite{pi} for details).

We are now ready to formulate the result.

\begin{theorem}\label{Lip-metric}
Let $(\Alg,\Hil_A,D_A)$ be an odd spectral triple on a unital $C^*$-algebra $A$. Assume  it satisfies the Lipschitz condition and let $\alpha:\Z\rightarrow\textup{Aut}(A)$ be a metrically equicontinuous action of $\Z$ on the spectral quantum metric space $(A,L_{D_A})$. Using the identity function $\iota:\Z\rightarrow\Z$ as a proper translation bounded function, we obtain an even spectral triple $(C_c(\Z,\Alg),\Hil\oplus\Hil,D)$ on $A\rtimes_\alpha\Z$ from Theorem \ref{constr}. The triple $(C_c(\Z,\Alg),\Hil\oplus\Hil,D)$ induces a Lip-metric on the state space of $A\rtimes_\alpha\Z$.
\end{theorem}
\begin{proof}
Note first that the action $\alpha$ satisfies the conditions of Theorem \ref{constr}. The triple $(\Alg,\Hil,D_A)$ is non-degenerate and so Theorem \ref{constr} shows that the crossed product triple is also non-degenerate. By the comments preceding the proof, it suffices to show that the image of
$$
\mathcal S=\{x\in C_c(\Z,\Alg):\|[x,D_A\otimes 1]\|\leq 1,\ \|[x,1\otimes M_\iota]\|\leq 1\}
$$
in $(A\rtimes_\alpha\Z)/\mathbb C1$ is totally bounded. Our hypothesis that the original triple induces a Lip-metric ensures that the set $\{a\in\Alg:\|a\|\leq 1,\ \|[\pi(a),D_A]\|\leq 1\}$ is totally bounded in $A$ (see \cite[Theorem 1.9]{R.MetricStateActions}). It then follows that, for each $N\in\mathbb N$, the set
$$
\mathcal S_N=\left\{\sum_{k=-N}^Na_k\lambda_k:a_k\in\Alg,\ \|[\pi(a_k),D_A]\|\leq 1\text{ and, when }k\neq 0,\ \|a_k\|\leq 1\right\}
$$
has a totally bounded image in $(A\rtimes_\alpha\Z)/\mathbb C1$. Indeed, given $\eps>0$, we can choose a finite $\eps/(2N+1)$-net $F\subset\{a\in\Alg:\|a\|\leq1,\ \|[a,D_A]\|\leq 1\}$ for the totally bounded set $\{a\in\Alg:\|a\|\leq 1,\ \|[a,D_A]\|\leq1\}$ and a finite set $F_0\subset\{a\in\Alg:\|[\pi(a),D_A]\|\leq 1\}$ whose image is an $\eps/(2N+1)$-net for the image of $\{a\in\Alg:\|[\pi(a),D_A]\|\leq 1\}$ in $A/\mathbb C1$. Then the image of
$$
\left\{\sum_{k=-N}^Na_k\lambda_k:a_k\in F,\ (k\neq 0),\ a_0\in F_0\right\}
$$
is an $\eps$-net for the image of $\mathcal S_N$ in $A/\mathbb C1$.

Now for each $x\in C_c(\Z,\Alg)$ write $x$ as a finite sum $x=\sum_{n\in\Z}x_n\lambda_n$. Then, for $\xi,\eta$ in the domain of $D_A$ and $n\in\Z$, we have
\begin{align*}
\ip{[\pi(x_n),D_A]\xi}{\eta}&=\ip{x\lambda_{-n}(D_A\otimes 1)(\xi\otimes\delta_e)}{\eta\otimes\delta_e}-\ip{x\lambda_{-n}(\xi\otimes\delta_e)}{(D_A\otimes 1)(\eta\otimes\delta_e)}\\
&=\ip{[x,D_A\otimes 1](\xi\otimes\delta_{-n})}{\eta\otimes\delta_e},
\end{align*}
and so
\begin{equation}\label{FC}
\|[\pi(x_n),D_A]\|\leq\|[x,D_A\otimes 1]\|,\quad n\in\Z.
\end{equation}
Since
$$
\ip{[x,1\otimes M_\iota](\xi\otimes\delta_{-n})}{\eta\otimes\delta_e}=-n\ip{x(\xi\otimes\delta_{-n})}{\eta\otimes\delta_e}=-n\ip{\pi(x_n)\xi}{\eta},\quad n\in\Z,\ \xi,\eta\in\Hil_A,
$$
for $n\neq 0$, we have $\|\pi(x_n)\|\leq|n|^{-1}\|[x,1\otimes M_\iota]\|\leq\|[x,1\otimes M_\iota]\|$. In conclusion, given $x=\sum_{n\in\Z}x_n\lambda_n\in\mathcal S$, we have $\sum_{k=-N}^Nx_k\lambda_k\in\mathcal S_N$ for each $N\in\mathbb N$.

Fix $\eps>0$. The `near diagonal cut-down' procedure of \cite{OR} shows that there exists $N\in\mathbb N$ such
that for any $x=\sum_{k\in\Z}x_k\lambda_k\in C_c(\Z,\Alg)$ with $\|[x,1\otimes M_\iota]\|\leq 1$, we have
$\|x-\sum_{k=-N}^Nx_k\lambda_k\|<\eps/2$.  Indeed,
\begin{equation}\label{Lip-Metric-equation}
\sum_{k=-N}^Nx_k\lambda_k=\sum_{k=-N}^N\sum_{m\in\Z}P_mxP_{m+k},
\end{equation}
where $P_m$ denotes the orthogonal projection from $\Hil\otimes\ell^2(\Z)$ onto $\Hil\otimes\mathbb C\delta_m$ and the second sum converges strongly.  The density result is then obtained by replacing positive integers by arbitrary integers in section 2 and 3 of \cite{OR}, which shows that for each $\mu>0$, there exists $N\in\mathbb N$ such that
$$
\|\sum_{|m-n|>N}P_mxP_n\|<\mu,
$$
for all $x\in C_c(\Z,\Alg)$
with $\|[x,1\otimes M_\iota]\|\leq 1$.
 As a consequence, an $\eps/2$-net for the totally bounded image of $\mathcal S_N$ in $(A\rtimes_\alpha\Z)/\mathbb C1$ then gives an $\eps$-net for the image of $\mathcal S$, completing the proof.
\end{proof}

\begin{remark}
It is worth noting what breaks down when we attempt to work with a more general group.  In \cite{OR}, Ozawa and
Rieffel show that given any finitely generated hyperbolic group $\Gamma$,  the spectral triple $(\mathbb
C\Gamma,\ell^2(\Gamma),M_{|\cdot|})$ on $C^*_r(\Gamma)$ arising from a length function $g\mapsto|g|$ on the group
gives $C^*_r(\Gamma)$ the structure of a quantum metric space. For $n\geq 0$, write $P_n$ for the orthogonal
projection on the finite dimensional subspace $\text{Span}\{\delta_g:|g|=n\}$ of $\ell^2(\Gamma)$. The near
diagonal cut-down procedure of \cite{OR} shows that, for $\eps>0$, there exists $N\in\mathbb N$ such that given
any $x\in\mathbb C\Gamma$ with $\|[x,M_{|\cdot|}]\|\leq 1$, one has
$$
\|x-\sum_{|m-n|\leq N}P_mxP_n\|<\eps.
$$
As Ozawa and Rieffel note, the operator $\sum_{|m-n|\leq N}P_mxP_n$ need not lie in $C^*_r(\Gamma)$. They proceed by splitting up each operator $x=\sum_{g\in\Gamma}\alpha_g\lambda_g\in\mathbb C\Gamma$ as
$$
x=\sum_{|m-n|\leq N}P_m(\sum_{|g|\leq K}\alpha_g\lambda_g)P_n+\sum_{|m-n|\leq N}P_m(\sum_{|g|>K}\alpha_g\lambda_g)P_n+\left(x-\sum_{|m-n|\leq N}P_mxP_n\right)
$$
for suitable $K,N\in\mathbb N$. The last term can be controlled by the near diagonal cut-down procedure
described above, while the first term only involves finitely many `Fourier-coefficients' of $x$ and so the image
of the first term in $C^*_r(\Gamma)/\mathbb C1$ over all those $x\in\mathbb C\Gamma$ with
$\|[x,M_{|\cdot|}]\|\leq 1$ is totally bounded. Ozawa and Rieffel control the second term by a Haagerup type
inequality. The absence of such an inequality when we replace the scalars $\alpha_g$ by operators $x_g$ on an
infinite dimensional Hilbert space is the obstruction preventing us from applying this strategy in the crossed
product context.  In the situation of Theorem \ref{Lip-metric}, where
$\Gamma=\Z$ is equipped with the identity function, so that the projections
$P_n$ are indexed by $\Z$ and not $\{n\in\Z:n\geq 0\}$, the near diagonal cut-down procedure coincides with the
process of approximating by finite sums of `Fourier-coefficients' and so the second term above vanishes. Another situation in which the above strategy works is the case when $\Gamma=\Z$ and we consider the triple on $C^*(\Z)$ arising from the word-length function given by the standard generating set $\{-1,1\}$,  as there we can exploit the fact that for any $n\in \bn$ there are precisely 2 elements of $\Z$ which have length $n$. We leave the details to the reader and note that  the problem for the word-lengths built on other generating sets on $\Z$, which do not have the property mentioned above, is still open.
\end{remark}

\begin{remark} \label{smoothgen}
Suppose we take $\Alg$ to be the Lipschitz algebra $\smooth$ in Theorem \ref{Lip-metric}. Under the hypotheses of that theorem, the resulting crossed-product triple $(C_c(\Z,\smooth),\Hil\oplus\Hil,D)$ induces a Lip-metric on the state space of $A\rtimes_\alpha\Z$.  The algebra $C_c(\Z,\smooth)$ is smaller than the Lipschitz algebra $\mathcal C^1(A\rtimes_\alpha\Z)$, but a similar argument to Theorem \ref{Lip-metric} can be used to show that under the same assumptions $\mathcal C^1(A\rtimes_\alpha\Z)$ also induces a Lip-metric on the state space of $A\rtimes_\alpha\Z$ and hence all (norm) dense subalgebras of $\mathcal C^1(A\rtimes_\alpha\Z)$ also induce a Lip-metric.  To see this, note that for any $x\in A\rtimes_\alpha\Z$ and $N\in\mathbb N$, (\ref{Lip-Metric-equation}) holds, where we define $x_k$ to be the `Fourier-coefficient' $x_k=E_A(x\lambda_{-k})$ and $E_A:A\rtimes_\alpha\Z\rightarrow A$ is the canonical conditional expectation, which can be defined spatially via $E_A(x)=P_0xP_0\in \pi(A)$. Thus, the cut-down procedure of Ozawa and Rieffel shows that for each $\eps>0$, there exists some $N\in\mathbb N$ with $\|x-\sum_{k=-N}^Nx_k\lambda_k\|<\eps/2$, whenever $x$ has $\|[x,1\otimes M_\iota]\|\leq 1$.  The estimate $\|[\pi(x_n),D_A]\|\leq\|[x,D_A\otimes 1]\|$ for all $n\in\Z$ is obtained in the same way as Theorem \ref{Lip-metric} and so the argument used in the proof of that theorem shows that $\mathcal C^1(A\rtimes_\alpha\Z)$ induces a Lip-metric on the state space of $A\rtimes_\alpha\Z$.
\end{remark}

\subsection{Construction of odd spectral triples and the iteration procedure}

As the last part of this section, we note that a similar procedure can also be used to produce odd spectral
triples on the crossed product $A\rtimes_{r,\alpha}\Gamma$ from an even triple on $A$, a suitable action $\alpha$
of $\Gamma$ on $A$ and a triple on $C^*_r(\Gamma)$ arising from a proper translation bounded function. In the
tensor product situation, given an even spectral triple
$$\left(\Alg,\Hil_{A,0}\oplus \Hil_{A,1},\begin{bmatrix}0&D_{A}\\D_A^*&0\end{bmatrix}\right)$$
on $A$ with the $\Z_2$-grading $\Hil_A=\Hil_{A,0}\oplus\Hil_{A,1}$, $\pi = \pi_0 \oplus \pi_1$ and an odd triple $(\Balg,\Hil_B,D_B)$ on another $C^*$-algebra $B$, we have a natural representation of the spatial tensor product $A\otimes B$ on $\Hil=(\Hil_{A,0}\otimes \Hil_B)\oplus(\Hil_{A,1}\otimes\Hil_B)$. The formalism of building Kasparov products as outlined in [Chapter 9, HR], shows that
$$
D=\begin{bmatrix}1\otimes D_B&D_{A}\otimes 1\\D_{A}^*\otimes 1&-1\otimes D_B\end{bmatrix},
$$
provides a Dirac-operator on $\Hil$ giving the odd spectral triple $(\Alg\odot\Balg,\Hil,D)$ on the spatial tensor product. This triple is non-degenerate when the original triples are non-degenerate and behaves well with respect to summability.

Let $\alpha$ be an action of a discrete group $\Gamma$ on $A$ and let $l$ be a proper translation bounded function on $\Gamma$ inducing the odd spectral triple $(\mathbb C\Gamma,\ell^2(\Gamma),M_l)$ on $C^*_r(\Gamma)$ as in Example \ref{GroupTriple}. Then, exactly in the same way as in the odd case, we have a diagonal representation of the reduced crossed product $A\rtimes_{r,\alpha}\Gamma$ on $\Hil=(\Hil_{A,0}\otimes\ell^2(\Gamma))\oplus(\Hil_{A,1}\otimes\ell^2(\Gamma))$. Provided $\alpha_g(\Alg)\subseteq \Alg$ for all $g\in\Gamma$ and the equicontinuity condition
$$
\sup_{g\in\Gamma}\|\pi_0(\alpha_g(x)) D_A - D_A \pi_1 (\alpha_g(x)) \|<\infty,\quad x\in\Alg,
$$
holds, where $\pi=\pi_0 \oplus \pi_1$ is the representation of $A$ on $\Hil_A$, the Dirac operator
$$
D=\begin{bmatrix}1\otimes M_l&D_A\otimes 1\\D_A^*\otimes 1&-1\otimes M_l\end{bmatrix},
$$
can be used to define an odd spectral triple $(C_c(\Gamma,\Alg),\Hil,D)$ on $A\rtimes_{r,\alpha}\Gamma$.  As in the odd case, this triple is non-degenerate when the triple on $A$ is non-degenerate and, when $\Gamma=\Z$ is equipped with the identity function, $D$ turns $A\rtimes_\alpha\Z$ into a quantum metric space when $D_A$ induces a quantum metric space structure on $A$.
The odd triple so obtained corresponds once again to a Kasparov product, this time of the Toeplitz extension with
$x= \left[\left(\Alg,\Hil_{A,0}\oplus \Hil_{A,1},\begin{bmatrix}0&D_{A}\\D_A^*&0\end{bmatrix}\right)\right]$ and the K-homology class of
$(C_c(\Z,\Alg),\Hil,D)$ is the image of $x \in K^0(A)$ under the boundary map in the Pimsner-Voiculescu sequence. One can also formulate and prove a counterpart of Proposition \ref{iterate} in this context. This allows us to iterate and thus provide satisfactory spectral triples on crossed products by metrically equicontinuous actions of $\Z^d$.

\begin{theorem} \label{Zd}
Let $(\Alg,\Hil_A,D_A)$ be an odd spectral triple on a unital $C^*$-algebra $A$. Assume  it satisfies the Lipschitz condition. Let $d \in \N$  and let $\alpha:\Z^d\rightarrow\textup{Aut}(A)$ be a metrically equicontinuous action of $\Z^d$ on the spectral quantum metric space $(A,L_{D_A})$. The iteration of the construction described in Theorem \ref{constr} and its even version discussed above leads to a spectral triple (odd if $d$ is even and even if $d$ is odd) on the crossed product $A\rtimes_\alpha\Z^d$, which induces a Lip-metric on the state space.
\end{theorem}

\begin{proof}
Let $\alpha_1,..., \alpha_d$ denote the coordinate $\Z$-actions of $\alpha$, and note that each $\alpha_i$ is metrically equicontinuous. The isomorphism
\[A\rtimes_\alpha\Z^d \approx \left( \cdots ((A \rtimes_{\alpha_1}\Z)\rtimes_{\alpha_2} \Z)\cdots \right)\rtimes_{\alpha_d} \Z\]
is well-known and easy to check.
Proposition \ref{iterate} and its counterpart for even triples imply via induction that the assumptions on the action required for the spectral triple construction in Theorem \ref{Lip-metric} are satisfied at each iterative step;
hence the result is a consequence of Theorem \ref{Lip-metric} and the discussion in the last two paragraphs.
\end{proof}

Note that if $d>1$ then the spectral triple on $A\rtimes_\alpha\Z^d$ produced iteratively as described in the last theorem  is not in general a spectral triple arising from a length function on $\Z^d$ via the construction from Theorem \ref{constr}.

\begin{remark} We are grateful to the referee for pointing out that these iterated spectral triples can be viewed as arising via proper translation bounded matrix-valued functions. To describe this we need first to introduce some further terminology.  Given a discrete group $\Gamma$ and $n\in\mathbb N$, consider a function $l:\Gamma\rightarrow M_{2n}^{sa}$ taking values in the self-adjoint $2n\times 2n$ matrices. We say that $l$ is \emph{translation bounded} if each translation function $l_g:\Gamma\rightarrow M_{2n}$ given by $l_g(x)=l(x)-l(g^{-1}x)$  ($g, x \in \Gamma$) is bounded; we say it is  \emph{proper} if $l(g)=0$ if and only if $g$ is the identity element of $\Gamma$, the set of all the eigenvalues of the matrices $\{l(g):g\in\Gamma\}$ is a discrete subset of $\mathbb R$ and no real number is the eigenvalue of infinitely many of the matrices $\{l(g):g\in\Gamma\}$ (compare with the definition given in Example \ref{GroupTriple}).

Note first that given such a translation bounded and proper matrix-valued function $l$, one obtains a spectral triple on $C^*_r(\Gamma)$ in a similar fashion to that described in Example \ref{GroupTriple}. Indeed, consider the representation $\lambda_{n}$ of $C^*_r(\Gamma)$ on $\ell^2(\Gamma)\otimes \mathbb C^{2n}$ given by amplifying the left regular representation, and define $M_l$ to be the densely defined operator of multiplication by $l(\cdot)$, i.e. $M_l(\delta_g\otimes\eta)=\delta_g\otimes l(g)\eta$ for $g\in\Gamma$ and $\eta\in \mathbb C^{2n}$. Then  $(\mathbb C\Gamma,\ell^2(\Gamma)\otimes\mathbb C^{2n},M_l)$ on $C^*_r(\Gamma)$ forms a spectral triple.

Now we describe how to view triples constructed on $A\rtimes_\alpha\Z^d$ in Theorem \ref{Zd} as arising from proper, translation bounded  matrix-valued functions.
To that end we  recursively define such "length" functions on $\Z^d$. Let $l^{(1)}:\Z\rightarrow M_2^{sa}$ be given by $l^{(1)}(n)=\begin{bmatrix}0&-in\\in&0\end{bmatrix}$. For $d>1$, define $l^{(d)}:\Z^d\rightarrow M_{2^{\lceil d/2\rceil}}$ by
$$
l^{(d)}(n_1,\cdots,n_d)=\begin{cases}l^{(d-1)}(n_1,\cdots,n_{d-1})+\begin{bmatrix}n_d1_{2^{d/2-1}}&0\\0&-n_d1_{2^{d/2-1}}\end{bmatrix},&d\text{ even};\\\begin{bmatrix}0&l^{(d-1)}(n_1,\cdots,n_{d-1})-in_d1_{2^{(d-1)/2}}\\{}l^{(d-1)}(n_1,\cdots,n_{d-1})^*+in_d1_{2^{(d-1)/2}}&0\end{bmatrix},&d\text{ odd},\end{cases}
$$
so that for example $l^{(2)}(n_1,n_2) = \begin{bmatrix}n_2&-in_1\\in_1&-n_2\end{bmatrix}$.
These are proper translation bounded functions on $\Z^d$. Further, given an odd spectral triple $(\mathcal A,\Hil,D_A)$ on a unital $C^*$-algebra $A$ satisfying the Lipschitz condition and a metrically equicontinuous action $\alpha:\Z^d\rightarrow\Aut(A)$, set $r=\lceil d/2\rceil$, and write
$$
\Hil=\Hil_A\otimes\ell^2(\Z^d)\otimes\mathbb C^{2^r}\cong(\Hil_A\otimes \ell^2(\Z^d)\otimes \mathbb C^{2^{r-1}})\oplus(\Hil_A\otimes \ell^2(\Z^d)\otimes \mathbb C^{2^{r-1}}).$$ Then the Dirac operator on $\Hil$ inducing a spectral triple on $A\rtimes_\alpha\Z^d$ arising from the iteration procedure of Theorem \ref{Zd} can be written in the form
$$
D=\begin{bmatrix}0&D_A\otimes 1_{\ell^2(\Z^d)\otimes \mathbb C^{2^{r-1}}}\\D_A\otimes 1_{\ell^2(\Z^d)\otimes \mathbb C^{2^{r-1}}}&0\end{bmatrix}+1_{\Hil_A}\otimes M_l,
$$
so arises in a similar fashion to the construction of spectral triples on crossed products by $\Z$, with matrix valued translation bounded functions $l^{(d)}$ in place of the scalar valued functions.

It would be interesting to find matrix valued proper translation bounded functions on other groups $\Gamma$, especially those inducing non-trivial elements of $K^1(C^*_r(\Gamma))$, which give rise to spectral triples on crossed products $A\rtimes_{r,\alpha}\Gamma$ satisfying the Lipschitz condition when $\alpha$ is a metrically equicontinuous action of $\Gamma$ on a spectral quantum metric space $A$.

\end{remark}

\section{Examples}

In this section we give some examples of metrically equicontinuous actions on $C^*$-algebras with spectral triples and hence obtain spectral triples on the resulting crossed product algebras. First we specialise Proposition \ref{equicontthm} to the commutative case, then examine the odometer action on the Cantor set to obtain spectral triples on the Bunce-Deddens algebras. We then consider actions of $\Z$ on non-commutative AF-algebras. Finally we return to the commutative situation and discuss generalized odometer actions of residually finite groups on the Cantor set. These give rise to the generalized Bunce-Deddens algebras of \cite{Orf}.

\subsection{Equicontinuous actions on compact spaces}

\begin{proposition} \label{equicontcomm}
Let $X$ be a compact metrizable space, $\Gamma$ a countable discrete group and $\alpha$ an action of $\Gamma$ on $X$. Write $\tilde{\alpha}$ for the action of $\Gamma$ on $C(X)$ induced by $\alpha$.  Suppose that $(\Alg, \Hil, D)$ is a spectral triple on $A=C(X)$ with faithful representation $\pi:C(X)\to B(\Hil)$ satisfying the Lipschitz condition and $\tilde{\alpha}$ acts smoothly with respect to this triple. Consider the following statements:
\begin{enumerate}[(i)]
\item $\tilde{\alpha}$ is metrically equicontinuous in the sense of Definition \ref{equicont}, i.e. for each $f\in \Alg$ there exists $M>0$ such that for each $g\in \Gamma$
\begin{equation} \label{HSWZ} \|[D,\pi(\tilde{\alpha}_{g}(f))]\|\leq M;\end{equation}
\item $\alpha$ is metrically equicontinuous, i.e.\ there exists an equivalent metric on $X$ for which all $\alpha_{g}$ ($g \in \Gamma$) are isometries.
\end{enumerate}
Then (ii)$\Longrightarrow$(i) in general, and (i)$\Longrightarrow$(ii) in the case that $\Alg=\smooth$.
\end{proposition}

\begin{proof}
The result is almost a special case of Proposition \ref{equicontthm}, noting that the condition $\Alg=\smooth$ ensures that the Lipschitz seminorm induced by $D$ is closed. Indeed we only need to show that if $d_1,d_2$ are two equivalent metrics on a compact space $X$, with equivalence constant $C$ (so that $C^{-1} d_1 \leq d_2 \leq C d_1$)  then the corresponding metrics $\rho_1,\rho_2$ on the state space of $C(X)$ are equivalent with the same equivalence constants (the converse is immediate).

Let us recall how the metric on the state space is constructed: we first define
\[ L_i(f) = \inf \, \{M:\forall_{x,y \in X} \; |f(x)-f(y)| \leq M d_i(x,y) \}, \;\; f \in C(X)\]
and then
\[ \rho_i (\phi, \psi) = \sup \{ |\phi(f)-\psi(f)|: f = f^* \in C(X),\ L_i(f) \leq 1\}.\]
Hence if $d_1$ and $d_2$ are equivalent metrics on $X$ with the equivalence constant $C$, then we have for any $f = f^*\in C(X)$
\[ C^{-1} L_2(f) \leq L_1(f) \leq C L_2(f),\]
and as $L$ is homogeneous (i.e.\ $L(\lambda f)= |\lambda| L(f)$, $f \in C(X), \lambda \in \bc$), for any $\phi, \psi$ in the state space of $C(X)$, we have
\[ C^{-1} \rho_1(\phi, \psi) \leq \rho_2(\phi, \psi) \leq C \rho_1(\phi, \psi).\]
Finally, it is easy to see that a metric on $X$ is invariant under $\alpha$ iff its extension to the state space of $C(X)$ is invariant under $\alpha$. This concludes the proof.
\end{proof}

It follows that our construction of  spectral triples provided by  Theorem \ref{Lip-metric} on $C(X)\rtimes_{\tilde{\alpha}}\Z$ can be applied only for actions $\alpha$ on $X$ which are at least equicontinuous.
In practice we will usually aim  to construct an isometric spectral triple. The motivating example is the construction of the spectral triple on irrational rotation algebras $A_{\Theta}$.  Rotations of the circle are certainly equicontinuous and, when $C(\T)$ is equipped with the usual spectral triple arising from the differentiation operator (which certainly satisfies the Lipschitz condition), a rotation action is smooth in the sense of Definition \ref{equicont}.  Thus Section \ref{sec2} gives a spectral triple on $C(\T)\rtimes_{\tilde{\alpha}} \Z\cong A_\Theta$, satisfying the Lipschitz condition; this procedure reproduces the usual Dirac operator on $A_{\Theta}$ (\cite{had}).

A general minimal equicontinuous action of $\Z$ on a compact metric space $X$ arises from rotation action on a compact abelian group (\cite[Theorem 2.42]{Kur}), i.e. $X$ can be equipped with the structure of an abelian group (so that in particular the addition is continuous) and the action is of the form $x\mapsto x+a$ for some $a\in X$.  In particular, when $X$ is a Cantor set, any minimal equicontinuous action is conjugate to an odometer action (\cite[Theorem 4.4]{Kur}).  Thus there are strong restrictions on what simple crossed products of commutative $C^*$-algebras admit crossed product triples in the fashion of Section \ref{sec2}.

\subsection{Bunce-Deddens Algebras}

Given a sequence $(m_n)_{n=1}^\infty$ of natural numbers with $m_n\geq 2$ for each $n$, consider the
Cantor set $X=\prod_{n=1}^\infty\Z_{m_n}$, where $\Z_{m_n}$ denotes the discrete group of integers with addition
modulo $m_n$.  The odometer action on $X$ is the homeomorphism $T:X\rightarrow X$ given by formal addition of
$(1,0,0,\cdots)$ with carry to the right.  This process continues infinitely so that
$T(m_1-1,m_2-1,m_3-1,\cdots)=(0,0,\cdots)$. With respect to the natural metric on $X$ given by the formula
$d(x,y)=\frac{1}{n}$, where $n$ is the number of the first coordinate at which $x$ and $y$ differ (and of course $d(x,y)=0$ if $x=y$), $T$ is an isometry and so the corresponding action of $\Z$ is metrically equicontinuous. The resulting crossed product $C(X)\rtimes_T\Z$ is isomorphic to the
Bunce-Deddens algebra $B_m$ with supernatural number $m=\prod_{n=1}^\infty m_n$ (\cite{da}, Chapter 8).

Let $A$ be a separable unital AF-algebra with a specified AF-filtration $(A_m)_{m=0}^\infty$ such that $A_0=\mathbb
C1_A$.  Given a faithful state $\phi$ on $A$,  let $\pi$ be the GNS-representation of $A$ on $\Hil=L^2(A,\phi)$
with cyclic vector $\xi$.  Define a sequence of pairwise orthogonal finite dimensional subspaces
$(E_m)_{m=0}^\infty$ of $\Hil$ corresponding to the filtration $(A_m)_{m=0}^\infty$ by $E_0=\pi(A_0)\xi=\mathbb
C\xi$, $E_m=\pi(A_m)\xi\ominus \pi(A_{m-1})\xi$ for $m\in \N$. Write $Q_m$ for the orthogonal projection from
$\Hil$ onto $E_m$. Given a sequence $(\lambda_m)_{m=0}^\infty$ with $|\lambda_m|\rightarrow\infty$ and
$\lambda_0=0$ one can define a Dirac operator $D=\sum_{m=0}^\infty \lambda_mQ_m$ on $\Hil$.
In \cite[Theorem2.1]{CI}, Christensen and Ivan show that given any separable unital AF-algebra $A$ and $\varepsilon>0$,
one can choose the sequence $(\lambda_m)$ to produce an $\varepsilon$-summable odd spectral triple $(\smooth, \pi, D)$
yielding a Lip-metric on the state space of $A$ with the Lipschitz algebra $\smooth$ containing $ \bigcup_{n=0}^\infty A_n$.  It is natural to take $\Alg=\bigcup_{n=0}^\infty A_n$ in this construction.  An automorphism $\theta:A\rightarrow A$ is smooth if $\theta(\Alg)=\Alg$, i.e. for each $n\in\N$, there exists $m\in\N$ with $\theta(A_n)\subseteq A_m$.

For the odometer action of $\Z$ on the Cantor set described above, write the Cantor set as the projective limit
\begin{equation}\label{bd.1}
\xymatrix{{\Z_{m_1}}&{\Z_{m_1}\times\Z_{m_2}}\ar[l]&{\Z_{m_1}\times\Z_{m_2}\times\Z_{m_3}}\ar[l]&{\cdots}\ar[l]&{\prod_{n=1}^\infty\Z_{m_n}=X}\ar[l]}.
\end{equation}
The odometer action arises as the limit of actions $T_n$ on $\prod_{r=1}^n\Z_{m_r}$, where $T_n$ is addition of $(1,0,\cdots,0)$
with carry to the right and $\phi_n(m_1-1,m_2-1,\cdots,m_n-1)=(0,0,\cdots,0)$. That is, by the identification
\begin{equation}\label{bd.2}
\prod_{r=1}^n\Z_{m_r}\stackrel{\cong}{\rightarrow} \Z_{m_1\cdots m_n},\quad (x_1,\cdots,x_n)\mapsto x_1+x_2m_1+x_3m_1m_2+\cdots+x_nm_1\cdots m_{n-1},
\end{equation}
$T_n$ is the action of addition of $1$ on $\Z_{m_1\cdots m_n}$ (which can be viewed simply as a cyclic permutation of the finite
set $\Z_{m_1\cdots m_n}$). Dualising, we obtain a natural filtration $(A_n)_{n=1}^\infty$ of $C(X)$ (set $A_0=\mathbb C1_X$) with respect to which the dual automorphism $\theta_T$ is smooth, i.e. $\theta_T(A_n)\subset\bigcup_{m=0}^\infty A_m$ for all $m$.

We can further see that $\theta_T$ satisfies the non-commutative equicontinuity property of Definition \ref{equicont} directly from the construction above. More than being smooth, $\theta_T$ satisfies $\theta_T(A_n)=A_n$ for all $n\in\mathbb N$. Since elements of $A_n$ commute with the projections $Q_m$ for $m>n$, it
follows that $\sup_{r\in\Z}\|[D,\pi({\theta_T}^r(x))]\|<\infty$ for $x\in\bigcup_{n=0}^\infty A_n$, so $\theta_T$ is metrically equicontinuous.  Provided we construct the Christensen-Ivan triple with respect to an invariant faithful state $\phi$ (i.e. $\phi\circ\theta_T=\phi$), such as the product of the uniform probability measures on each $\Z_{m_n}$, $\theta_T$ is isometric.  The invariance ensures that $\theta_T$ is spatially implemented by some unitary $U$ on $L^2(A,\phi)$, i.e.
$\pi(\theta_T(x))=U\pi(x)U^*$ for all $x\in A$.
Since $\theta_T(A_n)=A_n$ for all $n$, it follows that $U$ commutes with each of the
projections $Q_n$ and hence with the Dirac operator $D$.  Thus
\begin{equation}
\|[D,\pi(\theta_T(x))]\|=\|[D,U\pi(x)U^*]\|=\|[D,\pi(x)]\|,\quad \forall x\in\mathcal C^1(A).
\end{equation}
Thus, for every odometer action defined formally as addition by 1 on the Cantor set $X=\prod_{n=1}^\infty\Z_{m_n}$, there exists a Christensen-Ivan triple $(\bigcup_{n=0}^\infty A_n,\Hil,D)$, where  $\Hil$ is the GNS space of an faithful invariant state on $C(X)$ and satisfying the Lipschitz condition which induces an even spectral triple on the Bunce-Deddens algebra $B_m$ also satisfying the Lipschitz condition. We should note that the choice of Christensen-Ivan triple may depend quite strongly on the choice of odometer. For each $p>1$, one can construct such a triple which is $p$-summable. This critical value corresponds to the nuclear dimension of the
Bunce-Deddens algebra: as here $A_n \rtimes_{T_n}\Z = C(\prod_{i=1}^n \Z_{m_i}) \rtimes_{T_n}\Z \approx M_{\prod_{i=1}^n m_i}
\otimes C(\T)$, the Bunce-Deddens algebra is easily seen to be an A$\mathbb T$-algebra with nuclear dimension $1$ (\cite{nucdim}).

As all minimal equicontinuous actions of $\Z$ on the Cantor set $X$ are conjugate to odometer actions (see \cite[Theorem 4.4]{Kur} and Proposition \ref{Ellis} below), the Bunce-Deddens algebras are the only simple crossed products $C(X)\rtimes_\alpha\Z$ for which the procedure of Section \ref{sec2} gives us a spectral triple. It would be interesting to learn to what extent (if at all) it is possible to construct spectral triples on crossed products arising from other Cantor minimal systems (such as interval exchange transformations for which the corresponding crossed products were studied for example in \cite{Putnam}) using a different approach. Another possible extension, with a $\Z$-action replaced by an action of an arbitrary countable discrete group, is described in the last part of this section.

\subsection{Actions on AF-algebras}

The ideas above work also in the non-commutative situation. Let $\theta$ be an automorphism of a unital AF-algebra $A$. Say that $\theta$ \emph{fixes} an AF-filtration $(A_n)_{n=0}^\infty$ for $A$ if $\theta(A_n)=A_n$ for all $n\in \N$.  Just as above, this implies that $\theta$ is metrically equicontinuous with respect to any Christensen-Ivan triple constructed for $(A_n)_{n=0}^\infty$, and thus gives a spectral quantum metric space structure on the crossed product algebra $A\rtimes_\theta\Z$. Further if the Christensen-Ivan triple is constructed from an invariant faithful state, then $\theta$ is isometric with respect to this triple.  Since $\theta$ fixes the filtration $(A_n)_{n=0}^\infty$, the crossed product $A\rtimes_\theta\Z$ arises in a natural way as an inductive limit of $(A_n\rtimes_{\theta_n}\Z)_{n=1}^{\infty}$, where $\theta_n
= \theta|_{A_n}$. This inductive limit is compatible with the construction of spectral triples described in Section 2. Indeed, in
the notation used above we can see easily that $( \pi|_{A_n}, \bigoplus_{k\leq n}E_k, \xi)$ gives the GNS construction for the pair $(A_n, \phi_n)$,
where $\phi_n=\phi|_{A_n}$, and the Christensen-Ivan type Dirac operator for the (finite) filtration  $(A_k)_{k=0}^n$ in this
picture coincides with the restriction of the original $D$ to $\bigoplus_{k\leq n}E_k$. We record this as a proposition.

\begin{proposition} \label{indlimit}
Let $A$ be a unital separable AF-algebra with a fixed $AF$-filtration $(A_m)_{m=0}^\infty$ (with $A_0=
\mathbb{C} 1_A$), let $(\mathcal{A},  \pi, D)$ be a Christensen-Ivan type triple on $A$ and let $\theta \in
\textup{Aut}(A)$ fix the filtration. For each $n \in \mathbb{N}$ write $\theta_n=\theta|_{A_n}$ and denote the
corresponding Christensen-Ivan triple on $A_n=\bigcup_{m=0}^n A_m$ by $(A_n,  \pi_n, D_n)$. Let $(
\widetilde{A},\sigma , \widetilde{D})$ and  $( \widetilde{A}_n, \sigma_n, \widetilde{D}_n )$ denote the spectral triples
(respectively on $ A\rtimes_\theta\Z$ and on $A_n\rtimes_{\theta_n}\Z$) constructed according to Theorem
\ref{constr} of Section 2, with the respective Hilbert spaces denoted $\widetilde{\Hil}$ and $\widetilde{\Hil}_n$. Then there exists an increasing  sequence of  orthogonal projections $(P_n)_{n=0}^{\infty}$
in $B(\widetilde{\Hil})$, strongly convergent to $1_{B(\widetilde{\Hil})}$, such that the triple $( \widetilde{A}_n,
\sigma_n, \widetilde{D}_n)$ is unitarily equivalent to $(\widetilde{A}_n,
, P_n\sigma|_{A_n \rtimes_{\theta_n}\Z}P_n, P_n\widetilde{D}P_n)$.
\end{proposition}

\begin{remark}\label{weaklyfixed}
In the situation above, if the action $\theta$ does not fix the filtration $(A_n)_{n=0}^\infty$, but has the
property that for each $n\in\mathbb N$, there exists $m\in\mathbb N$, with $\theta^r(A_n)\subseteq A_m$ for all
$r\in\mathbb Z$, then $\sup_{r\in\Z}\|[D,\pi(\theta^r(x))]\|<\infty$ for all $x\in\bigcup_{n=0}^\infty A_n$,
allowing us to use Section 2 to produce a triple on the crossed product just as above. In this instance, however,
it is preferable to adjust the filtration by defining $B_n=C^*(\bigcup_{r\in\Z}\theta^r(A_n))$.  Each $B_n$ is
contained in some $A_m$, so it is finite-dimensional and moreover $\bigcup_{n=0}^\infty B_n=\bigcup_{n=0}^\infty A_n$. This new AF-filtration
is fixed by $\theta$. This procedure, at least when the Christensen-Ivan triple is constructed from a faithful invariant state,  gives another method of obtaining a new spectral triple for which $\theta$ is isometric (without passing to a larger Hilbert space as in Corollary \ref{spectralqms}).
\end{remark}

Product type automorphisms of UHF-algebras  provide the most natural non-commutative example of this phenomenon.   We begin with a general observation.

\begin{lemma}
Let $A,B$ be unital $C^*$-algebras, and let $\alpha \in \Aut(A \ot B)$ satisfy $\alpha(A\ot 1_B) = A \ot 1_B$. Assume that the
relative commutant of $A \ot 1_B$ in $A \ot B$ is equal to $1_A \ot B$ (this in particular implies that $A$ has trivial centre).
Then $\alpha$ is a product type automorphism.
\end{lemma}

\begin{proof}
It suffices to prove that $\alpha(1_A \ot B) \subset 1_A \ot B$ (the assumptions above hold also for $\alpha^{-1}$). By the
relative commutant condition we just need to show that if $a \in A$ and $b \in B$ then $\alpha(1_A \ot b) (a \ot 1_B) = (a \ot
1_B)\alpha(1_A \ot b)$. Substituting $a \ot 1_B$ by $\alpha(a' \ot 1_B)$ shows the last equality holds.
\end{proof}

Let $A$ be a UHF algebra. Then $A$ admits a filtration $(A_n)_{n=0}^{\infty}$ given by a sequence of tensor products of full matrix algebras of sizes
$k_i$, where $k_i > 1$ for $i \in \N$, so that $A_n$ is equal to $M_{\prod_{i=1}^n k_i}$. We call any such filtration a UHF-filtration.
The following fact is an immediate consequence of the last lemma.

\begin{proposition} \label{productaction}
Let $A$ be a UHF algebra with a UHF-filtration $(M_{\prod_{i=1}^n k_i})_{n=0}^{\infty}$.
Every automorphism of $A$ fixing this filtration  is of the form $\lim_{n \to \infty} \textup{Ad}(U_1 \ot
\cdots \ot U_n \ot 1 \ot \cdots)$ for a sequence of unitaries $(U_n)_{n=1}^{\infty}$ with $U_n \in M_{k_n}$. Conversely, each
such sequence of unitaries determines an automorphism of $A$ fixing the prescribed filtration.
\end{proposition}

Product type automorphisms of a UHF algebra have been studied by Lance, Bratteli and others
(c.f.\:\cite{Brat}). In particular the following criterion for innerness of a product automorphism is well known.

\begin{proposition}[Theorem 6.3 of \cite{book}] \label{outer}
A product automorphism is inner if and only if the sequence of unitaries $(U_1 \ot \cdots \ot U_n \ot 1 \ot
\cdots)_{n=1}^{\infty}$ is (norm) convergent in $A$.
\end{proposition}

Combining these results and the construction in \cite{CI} we obtain the following result.

\begin{proposition}\label{outersimple}
Let $A$ be a UHF algebra and suppose $\alpha \in \Aut(A)$ fixes a UHF-filtration $(A_n)$. Then
$A \semidir_{\alpha} \Z$ allows spectral triples which are $p$-summable for $p \in (1,3)$ or $p=5$ and which also induce Lip-metrics
on the state space of $A \semidir_{\alpha} \Z$.
\end{proposition}
\begin{proof}
For $p \in (0,2)$ or $p=4$ there is a $p$-summable Christensen-Ivan spectral triple $(\Alg,\pi ,D)$  on $A$ with  $\Alg = \bigcup_{n \in \N} A_n$, inducing a Lip-metric (\cite{CI} Theorem 3.1). These triples act in the GNS-representation $\pi$ of the trace on $A$. Since the trace
is invariant under $\alpha$ it follows, as remarked before, that $\alpha$ is isometric with respect to $D$, so that Theorems \ref{Lip-metric} and \ref{constr} apply to produce the desired triple.
\end{proof}

By a result of Kishimoto (\cite{Kish}) we know that if $A$ is a simple $C^*$-algebra and $\alpha \in \Aut(A)$
an automorphism such that $\alpha^m$ is outer for each $m \in \mathbb{Z} \setminus\{0\}$, then $A \semidir_{\alpha} \Z$
is simple. It is possible to choose a sequence of unitaries $(U_i)_{i=1}^{\infty}$ in such a way that the resulting product-type automorphism $\alpha$ satisfies the above condition.

The remarks preceding Proposition \ref{indlimit} imply that the crossed product $A \semidir_{\alpha} \Z$ is an inductive limit of the algebras
of the form $M_{m_1\cdots m_n}(C(\mathbb{T}))$. In fact Bratteli in \cite{Brat} describes these inductive limits explicitly, with
the connecting maps given in terms of the spectra of the unitaries defining the product action. Note that though the building
blocks here are identical to these in the Bunce-Deddens algebra case, in fact the identifications $C(\Z_{m_1,\ldots, m_n})
\rtimes_{T_n} \Z$ and $M_{m_1\ldots m_n} \rtimes_{\alpha_{U_1\cdots U_n}} \Z$ with $M_{m_1 \cdots m_n} \ot C(\T)$ are quite
different, and the resulting spectral triples on the limit A$\T$-algebras have a very different character.

In the light of the classification of minimal equicontinuous actions on the Cantor set, the following question is natural.
\begin{question}
Let $A$ be a UHF-algebra and $\alpha$ an automorphism of $A$ such that $\alpha^m$ is outer for each $m\neq 0$, so that $A\rtimes_\alpha\Z$ is simple by \cite{Kish}. If in addition the induced action on the state space $S(A)$ is equicontinuous, must $\alpha$ be conjugate to a product type automorphism?
\end{question}

\subsection{Generalized Bunce-Deddens algebras} In this final subsection we look at more general groups $\Gamma$ acting equicontinuously on the Cantor set. Provided $\Gamma$ is equipped with a length function, we obtain spectral triples on the resulting crossed product, but in general we do not know whether these triples satisfy the Lipschitz condition.

Let $\Gamma$ be a countable discrete  group and let $(\Gamma_n)_{n=1}^\infty$ be a strictly decreasing sequence of finite index subgroups of $\Gamma$.  For each $n$, write $X_n$ for the left coset space $\Gamma/\Gamma_n$ so we have an action $\Gamma\curvearrowright X_n$ of $\Gamma$ on the finite set $X_n$ given by left multiplication.  The inclusions $\Gamma_n\supset \Gamma_{n+1}$ induce $\Gamma$-equivariant maps $X_n\leftarrow X_{n+1}$.  Let $X$ be the projective limit in the category of $\Gamma$-spaces of
$$
X_1\leftarrow X_2\leftarrow\cdots\leftarrow X_n\leftarrow\cdots.
$$
Write $\iota_n$ for the induced $\Gamma$-equivariant map $X\rightarrow X_n$. By equipping $X$ with the compatible metric $d(x,y)=1/n$, where $n\in \N$ is maximal such that $\iota_n(x)=\iota_n(y)$ (with $d(x,y)=0$ when $x=y$) we see that the action of $\Gamma$ on $X$ is equicontinuous.  Since $\Gamma$ is transitive on each $X_n$, each orbit of $\Gamma$ is dense in $X$, i.e. the action is minimal. If the sequence $(\Gamma_n)_{n\in \N}$ witnesses the residual finiteness of $\Gamma$ (i.e. $\bigcap_{n\in \N}\Gamma_n=\{e\}$), then the action is free.  Terminology for these actions varies: F.~Krieger refers to them as $\Gamma$-odometers in \cite{Kri}, while \cite{CP} refers to them as $\Gamma$\emph{-subodometers}, reserving the terminology $\Gamma$-odometers for the case when all the subgroups $\Gamma_n$ are normal, so that $X$ carries a group structure as a profinite completion of $\Gamma$. We will adopt the latter convention.

It is undoubtably well known to experts that every minimal equicontinuous action of a countable discrete group arises in this way. We have been unable to find this statement in the literature, so we include a proof, which is a simple modification of K\r{u}rka's account of the classical result, sometimes attributed to Ellis, that every minimal equicontinuous action of $\mathbb Z$ on a Cantor set is conjugate to an odometer action \cite[Theorem 4.4]{Kur}.
\begin{proposition} \label{Ellis}
Let $\alpha:\Gamma\curvearrowright X$ be a minimal equicontinuous action of a countable discrete group $\Gamma$ on the Cantor set $X$. Then $\alpha$ is conjugate to a $\Gamma$-subodometer. If the action $\alpha$ is free, then $\Gamma$ is necessarily residually finite.
\end{proposition}

\begin{proof}
Let $d$ be a compatible metric on $X$ invariant under the action of $\Gamma$.
For $\eps>0$, define an equivalence relation $\approx_\eps$ on $X$ by $x\approx_\eps y$ if and only if there is a finite sequence $x_0,\cdots,x_n$ in $X$ with $x_0=x$, $x_n=y$ and $d(x_i,x_{i+1})<\eps$ for $i=0,\cdots,n-1$. The equivalence classes of $\approx_\eps$ form a finite partition of $X$ into clopen subsets.

There exists $\eps_1>0$ such that $\approx_{\eps_1}$ has at least two equivalence classes. Let $\mathcal V_1$ be the set of these equivalence classes. By construction $\Gamma$ permutes these equivalence classes and the action of $\Gamma$ on $\mathcal V_1$ is transitive as the original action on $X$ was minimal.  Fix a distinguished element $V_1\in\mathcal V_1$ and let $\Gamma_1$ denote the stabiliser of $V_1$ under the permutation action so that $[\Gamma:\Gamma_1]=|\mathcal V_1|<\infty$. In this way the actions of $\Gamma$ on $\Gamma/\Gamma_1$ and $\mathcal V_1$ are conjugate.

Now find $\eps_2<\eps_1/2$ small enough so that each member of $\mathcal V_1$ is subdivided into at least two elements of the  equivalence classes $\mathcal V_2$ of $\approx_{\eps_2}$. Fix a distinguished class $V_2\in\mathcal V_2$ with $V_2\subset V_1$ and let $\Gamma_2$ denote the stabiliser of $V_2$ under the permutation action of $\Gamma$ on $\mathcal V_2$. In this way $\Gamma_2$ is a finite index subgroup of $\Gamma$ which is contained in $\Gamma_1$ and the actions of $\Gamma$ on $\Gamma/\Gamma_2$ and $\mathcal V_2$ are conjugate.

Continuing in this fashion we obtain sequences  $(\eps_n)_{n=1}^\infty$, $(\Gamma_n)_{n=1}^\infty$, $(\mathcal V_n)_{n=1}^\infty$ and $(V_n)_{n=1}^\infty$.  By construction, the  system $\bigcup_{n \in \N}\mathcal V_n$ provides a basis for the topology on $X$ and so the action $\alpha$ of $\Gamma$ on $X$ is conjugate to the $\Gamma$-subodometer action we have constructed. Note that the unique point in $\bigcap_{n\in \N}V_n$ has the stabiliser equal to $\bigcap_{n\in \N}\Gamma_n$. Thus if the subodometer action (equivalently, $\alpha$) is free, then $\bigcap_{n\in \N}\Gamma_n=\{e\}$, so that $\Gamma$ is residually finite. This completes the proof.
\end{proof}

The $C^*$-algebras arising as crossed products by free $\Gamma$-odometers for amenable residually finite $\Gamma$ are simple and nuclear and have recently been studied by Orfanos \cite{Orf} and Carri\'on \cite{Car}. In particular Carri\'on has classified such crossed products by their $K$-theory under the assumption that $\Gamma$ is a central extension of a finitely generated abelian group by another finitely generated abelian group. The group $\mathbb Z^d$ belongs to this class, and by the iteration procedure from Theorem \ref{Zd} we obtain spectral triples which satisfy the Lipschitz condition on the crossed products arising from $\mathbb Z^d$-odometers.

\vspace*{0.2cm} \noindent \textbf{Acknowledgment.}  AS and SW would like to thank Thierry Giordano for his helpful comments on equicontinuous actions. We would like to thank the referee for helpful comments and suggestions for further research directions.

\def\ocirc#1{\ifmmode\setbox0=\hbox{$#1$}\dimen0=\ht0 \advance\dimen0
  by1pt\rlap{\hbox to\wd0{\hss\raise\dimen0
  \hbox{\hskip.2em$\scriptscriptstyle\circ$}\hss}}#1\else {\accent"17 #1}\fi}
\providecommand{\bysame}{\leavevmode\hbox to3em{\hrulefill}\thinspace}

\end{document}